\documentclass[smallextended,envcountsect]{svjour3} 
\usepackage{hyperref, comment}
\usepackage{latexsym,array,delarray,amsthm,amssymb,epsfig,setspace,tikz,amsmath,enumerate,mathrsfs,graphicx}

\usepackage{natbib}

\def\makeheadbox{\relax}

\theoremstyle{plain}
\newtheorem{thm}{Theorem}

\newtheorem{lem}{Lemma}

\newtheorem{prop}{Proposition}

\newtheorem{cor}{Corollary}

\newtheorem*{thm*}{Theorem}

\newtheorem*{prop*}{Proposition}

\newtheorem*{cor*}{Corollary}

\theoremstyle{definition}
\newtheorem{defn}{Definition}
\newtheorem{ex}{Example}

\newtheorem{conj}{Conjecture}

\theoremstyle{remark}

\newcommand{\zz}{\mathbb{Z}}
\newcommand{\nn}{\mathbb{N}}

\newcommand{\rr}{\mathbb{R}}

\newcommand{\cf}{\mathcal{F}}

\newcommand{\ind}{\mbox{$\perp \kern-5.5pt \perp$}}

\newcommand{\indep}{\rotatebox[origin=c]{90}{$\models$}}

\begin{document}

\title{Gaussian graphical models with toric vanishing ideals \thanks{Pratik Misra and Seth Sullivant were  partially supported 
by the US National Science Foundation (DMS 1615660).}}

\author{Pratik Misra         \and
Seth Sullivant}

\institute{Pratik Misra (corresponding author) \at
Department of Mathematics \\
North Carolina State University\\
2108 SAS Hall Box 8205\\
Raleigh, NC 27695, USA\\
\email{pmisra@ncsu.edu}
\and
Seth Sullivant \at
Department of Mathematics \\
North Carolina State University\\
2108 SAS Hall Box 8205\\
Raleigh, NC 27695, USA\\
\email{smsulli2@ncsu.edu} 
}


\maketitle

\begin{abstract}
Gaussian graphical models are semi-algebraic subsets of the cone of positive definite covariance matrices. 
They are widely used throughout natural sciences, computational biology and many other fields. 
Computing the vanishing ideal of the model gives us an implicit description of the model. In this paper, we resolve two conjectures given by Sturmfels and Uhler. 
In particular, we characterize those graphs for which the vanishing ideal of the
Gaussian graphical model is generated in degree $1$ and $2$.  These turn out to be
the Gaussian graphical models whose ideals are toric ideals, and the resulting
graphs are the $1$-clique sums of complete graphs.  
\end{abstract}


\maketitle

\section{Introduction}
Any positive definite $n\times n$ matrix $\Sigma$ can be seen as the covariance matrix 
of a multivariate normal distribution in $\mathbb{R}^n$. The inverse matrix $K=\Sigma^{-1}$ is called the \emph{concentration} matrix of the distribution, which is also positive definite. The statistical models where the concentration matrix $K$ can be written as a linear combination of some fixed linearly independent symmetric matrices $K_1,K_2,...,K_d$ are called \emph{linear concentration} models.

Let $\mathbb{S}^n$ denote the vector space of real symmetric matrices and let $\mathcal{L}$ be a linear subspace of $\mathbb{S}^n$ generated by $K_1,K_2,...,K_d$. The set $\mathcal{L}^{-1}$ is defined as 
\[
\mathcal{L}^{-1}=\{\Sigma \in \mathbb{S}^n : \Sigma^{-1}\in \mathcal{L}\}.
\]
The homogeneous ideal of all the polynomials in 
$\mathbb{R}[\Sigma]=\mathbb{R}[\sigma_{11},\sigma_{12},...,\sigma_{nn}]$ that vanish 
on $\mathcal{L}^{-1}$ is denoted by $P_\mathcal{L}$. Note that $P_\mathcal{L}$ is prime because
it is the vanishing ideal of $\mathcal{L}^{-1}$, which is the image of the irreducible variety
$\mathcal{L}$ under the rational inversion map.  
In this paper, we study the problem of finding a generating set of $P_\mathcal{L}$ for the special case of Gaussian graphical models.

Gaussian graphical models are used throughout the natural sciences and 
especially in computational biology as seen in \cite{Koller(2009)}, \cite{Lauritzen(1996)}. 
These models explicitly capture the statistical relationships between the 
variables of interest in the form of a graph. 
The undirected Gaussian graphical model is obtained when the subspace 
$\mathcal{L}$ of $\mathbb{S}^n$ is defined by the vanishing of some off-diagonal 
entries of the concentration matrix $K$. We fix a graph $G=([n],E)$ with vertex set 
$[n]=\{1,2,...,n\}$ and edge set $E$, which is assumed to contain all self loops. 
The subspace $\mathcal{L}$ is generated by the set $\{K_{ij}|(i,j)\in E\}$ of 
matrices $K_{ij}$ with $1$ entry at the $(i,j)^{th}$ and $(j,i)^{th}$ position 
and $0$ in all other positions. We denote the ideal $P_\mathcal{L}$ as $P_G$ in this model.

One way to compute $P_G$ is to eliminate the entries of an 
indeterminate symmetric $n \times n$ matrix $K$ from the following system of equations:
\[
\Sigma \cdot K = Id_n,\hspace{.5cm}  K\in \mathcal{L},
\]
where $Id_n$ is the $n\times n$ identity matrix.  However, this elimination is
computationally expensive, and we would like methods to identify generators of $P_G$ directly
in terms of the graph.

Various methods have been proposed for
finding some generators in the ideal $P_G$ and for trying to build $P_G$ from smaller ideals associated
to subgraphs.  These approaches are based on separation criteria in the graph $G$.  

\begin{defn}
Let $G=(V,E)$ be a graph.    
\begin{itemize}
\item  A set $C \subseteq V$ is called a \emph{clique} of $G$ if the subgraph induced by $C$ is a complete graph.
\item  Let $A,B,$ and $C$ be disjoint subsets of the vertex 
set of $G$ with $A\cup B\cup C=V$.  Then $C$ \emph{separates} $A$ and $B$ 
if for any $a\in A$ and $b\in B$, any path from $a$ to $b$ passes through a 
vertex in $C$. 
\item   The graph $G$ is said to be a \emph{c-clique sum} of smaller graphs $G_1$ and $G_2$ if there exists a partition $(A,B,C)$ of its vertex set such that 
\begin{enumerate}[i)]
\item $C$ is a clique with $|C|=c$, 
\item $C$ separates $A$ and $B$,
\item $G_1$ and $G_2$ are the subgraphs induced by $A\cup C$ and $B\cup C$ respectively. 
\end{enumerate}
In the case that $G$ is a $c$-clique sum, we call the  corresponding 
partition $(A,B,C)$ a \emph{c-clique partition} of $G$.
\end{itemize}

\end{defn}

If $G$ is a c-clique sum of $G_1$ and $G_2$, the ideal
\begin{eqnarray}\label{eqnarray:c-clique}
P_{G_1} + P_{G_2} + \langle (c+1)\times (c+1) \text{-minors of } \Sigma _{A\cup C, B\cup C} \rangle
\end{eqnarray}
is contained in $P_G$. 
Here $\Sigma_{A\cup C,B\cup C}$ denotes the submatrix of $\Sigma$ obtained by taking all rows indexed
by $A \cup C$ and columns indexed by $B \cup C$, and so   
\[
\langle (c+1)\times (c+1) \text{-minors of } \Sigma _{A\cup C, B\cup C} \rangle
\]
is the conditional independence ideal associated to the conditional independence statement $A \ind B | C$.
Though the ideal (\ref{eqnarray:c-clique}) fails to  equal $P_G$, (or even have the same radical as that of $P_G$) 
for $c\geq 2$,
\cite{Sturmfels(2010)}  conjectured it to be equal to $P_G$ for $c=1$.

\begin{conj}[\cite{Sturmfels(2010)}]\label{conj:false}
Let $G$ be a 1-clique sum of two smaller graphs $G_1$ and $G_2$. If $(A,B,C)$ is the 1-clique partition of $G$ where $G_1$ and $G_2$ are the subgraphs induced by $A\cup C$ and $B\cup C$ respectively, then
\begin{eqnarray*}
P_G= P_{G_1} + P_{G_2} + \langle 2\times 2 \text{-minors of } \Sigma _{A\cup C, B\cup C} \rangle.
\end{eqnarray*}
\end{conj}

In Section \ref{sec:conjfalse}, we give counterexamples to this conjecture, and even a natural
strengthening of it.  However, the motivation for Conjecture \ref{conj:false}
was to use it as a tool to prove a different conjecture characterizing the 
graphs for which the vanishing ideal $P_G$ is generated in degree $\leq 2$.  To
explain the details of this conjecture we need some further notions.

Let $X=(X_1,X_2,\ldots,X_n)$ be a Gaussian random vector. 
If $A,B,C\subseteq [n]$ are pairwise disjoint subsets, 
then from Proposition 4.1.9 of  \cite{Sullivant(2018)} we know that 
$X_A$ is conditionally independent of $X_B$ given $X_C$ (i.e $A\indep B|C$) if and only if 
the submatrix $\Sigma_{A\cup C, B\cup C}$ of the covariance matrix $\Sigma$ has rank $|C|$. 
The \emph{Gaussian conditional independence ideal} for 
the conditional independence statement $A\indep B|C$ is given by 
\[ 
J_{A\indep B|C}=\langle (|C|+1)\times (|C|+1) \text{ minors of } \Sigma_{A\cup C,B\cup C} \rangle.
\]
If $G$ is an undirected graph and $(A,B,C)$ is a partition with $C$ 
separating $A$ from $B$, then the conditional independence statement $A\indep B|C$ 
holds for all multivariate normal distributions where the covariance matrix 
$\Sigma$ is obtained from $G$ (by the global Markov property). 
The \emph{conditional independence ideal} for the graph $G$ is defined by 
\[
CI_G= \sum_{A\indep B|C \text{ holds for }G}J_{A\indep B|C}.
\]

\begin{prop}
For any given graph $G$, $CI_G\subseteq P_G$.
\end{prop}

\begin{proof}
As the rank of the submatrices $\Sigma_{A\cup C,B\cup C}$ of the covariance matrix $\Sigma$ is $|C|$ for all partitions $(A,B,C)$ of $G$, the generators of $CI_G$ vanish on the matrices in $\mathcal{L}^{-1}$.
\end{proof}

\begin{defn}
A graph $G$ is called a \emph{1-clique sum of complete graphs} if there exists a partition $(A,B,C)$ of its vertex set such that 

\begin{enumerate}[i)]
\item  $|C|=1$, 
\item $C$ separates $A$ and $B$, 
\item the subgraphs induced by $A\cup C$ and $B\cup C$ are either complete graphs or 1-clique sum of complete graphs. 
\end{enumerate}
\end{defn}

The second conjecture in \cite{Sturmfels(2010)} which we prove in this paper is as follows:

\begin{thm}(Conjecture 4.4, \cite{Sturmfels(2010)})\label{thm: theConjecture}
The prime ideal $P_G$ of an undirected Gaussian graphical model is 
generated in degree $\leq$ 2 if and only if each connected component 
of the graph $G$ is a 1-clique sum of complete graphs. 
\end{thm}

The ``only if" part of the conjecture is proved in 
\cite{Sturmfels(2010)}.  That is, it is shown there that a graph that is not the $1$-clique sum of complete graphs
must have a generator of degree $\geq 3$.  Such a generator comes from a conditional
independence statement with $\#C \geq 2$.  

For 1-clique sum of complete graphs, 
the conditional independence ideal can be written as
\[
CI_G= \langle \bigcup_{\substack{(A,B,C)\in C_1(G)}} 2\times 2 \text{ minors of } \Sigma_{A\cup C, B\cup C} \rangle,
\]
where $C_1(G)$ denotes the set of all 1-clique partitions of $G$. 
In this paper, our main result will be a proof that $CI_G = P_G$  when 
$G$ is a 1-clique sum of complete graphs. 

The expression ``$1$-clique sum of complete graphs'' is somewhat cumbersome.
We use the alternate expression \emph{block graphs} for most of the paper,
as that is a commonly used name in the literature.  One important property
of block graphs is that there is a unique locally shortest path between any pair
of vertices in a connected component of a block graph.

\begin{ex}
We illustrate the structure of Theorem \ref{thm: theConjecture} with an example. 
Let $G=([6],E)$ be the block graph as shown in Figure \ref{fig: Calculating CI_G}. 
This block graph $G$ has four 1-clique partitions as follows:

Partition 1: $A=\{1,2\},B=\{4,5,6\},C=\{3\},$
Partition 2: $A=\{1,2,3\},B=\{5,6\},C=\{4\}$

Partition 3: $A=\{1,2,3,5\},B=\{6\},C=\{4\},$
Partition 4: $A=\{1,2,3,6\},B=\{5\},C=\{4\}.$

The associated matrices are as follows:

\[
\text{For 1}:
\Sigma_{A\cup C,B\cup C}=
\begin{bmatrix}
    \sigma_{13}       & \sigma_{14} & \sigma_{15} & \sigma_{16} \\
    \sigma_{23}       & \sigma_{24} & \sigma_{25} & \sigma_{26} \\
    \sigma_{33}       & \sigma_{34} & \sigma_{35} & \sigma_{36} 
    \end{bmatrix},
2:
\Sigma_{A\cup C,B\cup C}=
\begin{bmatrix}
    \sigma_{14}       & \sigma_{15} & \sigma_{16}  \\
    \sigma_{24}       & \sigma_{25} & \sigma_{26}  \\
    \sigma_{34}       & \sigma_{35} & \sigma_{36}  \\
    \sigma_{44}       & \sigma_{45} & \sigma_{46}     
    \end{bmatrix}  
\]
    
\[
\hspace{-0.1cm}3:\Sigma_{A\cup C,B\cup C}=
\begin{bmatrix}
    \sigma_{14}       & \sigma_{16}   \\
    \sigma_{24}       & \sigma_{26}   \\
    \sigma_{34}       & \sigma_{36}   \\
    \sigma_{44}       & \sigma_{46}   \\
    \sigma_{45}       & \sigma_{56}    
    \end{bmatrix},
\hspace{1.7cm}4:\Sigma_{A\cup C,B\cup C}=
\begin{bmatrix}
    \sigma_{14}       & \sigma_{15}   \\
    \sigma_{24}       & \sigma_{25}   \\
    \sigma_{34}       & \sigma_{35}   \\
    \sigma_{44}       & \sigma_{45}   \\
    \sigma_{46}       & \sigma_{56}    
    \end{bmatrix}.
\]

The ideal $CI_G = P_G$ is the ideal generated by the 
$2\times 2$ minors of all four matrices: 
\begin{eqnarray*}
&CI_G=\langle \sigma_{13}\sigma_{24}-\sigma_{14}\sigma_{23},\sigma_{13}\sigma_{25}-\sigma_{15}\sigma_{23},\sigma_{13}\sigma_{26}-\sigma_{16}\sigma_{23},\sigma_{14}\sigma_{25}-\sigma_{15}\sigma_{24},\sigma_{23}\sigma_{34}-\sigma_{24}\sigma_{33},\\
&\sigma_{23}\sigma_{35}-\sigma_{25}\sigma_{33},\sigma_{23}\sigma_{36}-\sigma_{26}\sigma_{33},\sigma_{24}\sigma_{35}-\sigma_{25}\sigma_{34},\sigma_{24}\sigma_{36}-\sigma_{26}\sigma_{34},\sigma_{25}\sigma_{36}-\sigma_{26}\sigma_{35},\\
&\sigma_{13}\sigma_{34}-\sigma_{14}\sigma_{33},\sigma_{13}\sigma_{35}-\sigma_{15}\sigma_{33},\sigma_{13}\sigma_{36}-\sigma_{16}\sigma_{33},\sigma_{14}\sigma_{35}-\sigma_{15}\sigma_{34},\sigma_{14}\sigma_{36}-\sigma_{16}\sigma_{34},\\
&\sigma_{15}\sigma_{36}-\sigma_{16}\sigma_{35},\sigma_{14}\sigma_{45}-\sigma_{15}\sigma_{44},\sigma_{14}\sigma_{46}-\sigma_{16}\sigma_{44},\sigma_{15}\sigma_{46}-\sigma_{16}\sigma_{45},\sigma_{24}\sigma_{45}-\sigma_{25}\sigma_{44},\\
&\sigma_{24}\sigma_{46}-\sigma_{26}\sigma_{44},\sigma_{25}\sigma_{46}-\sigma_{26}\sigma_{45},\sigma_{34}\sigma_{45}-\sigma_{35}\sigma_{44},\sigma_{34}\sigma_{46}-\sigma_{36}\sigma_{44},\sigma_{35}\sigma_{46}-\sigma_{36}\sigma_{45},\\
&\sigma_{14}\sigma_{56}-\sigma_{16}\sigma_{45},\sigma_{24}\sigma_{56}-\sigma_{26}\sigma_{45},\sigma_{34}\sigma_{56}-\sigma_{36}\sigma_{45},\sigma_{44}\sigma_{56}-\sigma_{46}\sigma_{45},\sigma_{14}\sigma_{56}-\sigma_{15}\sigma_{46},\\
&\sigma_{24}\sigma_{56}-\sigma_{25}\sigma_{46},\sigma_{34}\sigma_{56}-\sigma_{35}\sigma_{46},\sigma_{44}\sigma_{56}-\sigma_{45}\sigma_{46},
\sigma_{14}\sigma_{26}-\sigma_{16}\sigma_{24},\sigma_{15}\sigma_{26}-\sigma_{16}\sigma_{25}
\rangle.
\end{eqnarray*}

\begin{figure}
\begin{center}
\begin{tikzpicture}
  [scale=.6,auto=left,every node/.style={circle,fill=blue!20}]
  \node (n4) at (9,0)  {4};
  \node (n1) at (0,0) {1};
  \node (n2) at (3,3)  {2};
  \node (n3) at (5.5,0)  {3};
  \node (n5) at (12,3)  {5};
  \node (n6) at (12,-3)  {6};
\foreach \from/\to in {n1/n2,n2/n3,n3/n4,n1/n3,n4/n5,n4/n6}
    \draw (\from) -- (\to);
\end{tikzpicture}
\caption{\label{fig: Calculating CI_G} }
\end{center}
\end{figure}

\end{ex}

The history of trying to characterize constraints on the covariance matrices in Gaussian
graphical models goes back to \cite{Kelley(1935)} and the discovery of the pentad constraints
in the  factor analysis model.  Since then, the study of the constraints on Gaussian graphical 
models has seen many results including the deeper study of the factor analysis model in \cite{Drton(2007)},
the study of directed graphical models and characterization of tree models in \cite{Sullivant(2008)},
and the complete characterization of the determinantal constraints that apply to Gaussian graphical models in \cite{Sullivant(2010)}.

The study of the generators of the ideals $P_G$ is an important problem
for constraint-based inference for inferring the structure of the underlying graph
from data.  Elements of the vanishing ideal are tested to determine if the graph
has certain underlying features, which are then used to reconstruct the entire graph.
A prototypical example of this method is the TETRAD procedure in \cite{Spirtes(2000)}
which specifically tests the degree $2$ generators (tetrads) of the vanishing ideals of
Gaussian graphical models for directed graphs.  Our main result in this paper gives a characterization
of which undirected graphs the tetrads are sufficient to characterize all distributions
from the model, and is a key structural result for trying to use constraint based 
inference for undirected Gaussian graphical models.   Developing characterizations of the
vanishing ideals of Gaussian graphical models by higher order constraints 
(for example, determinantal constraints in \cite{Drton(2008)} and \cite{Sullivant(2010)} )
has the potential to extend constraint-based inference beyond tetrad constraints.

This paper is organized as follows. We give two counterexamples to Conjecture \ref{conj:false} 
in Section \ref{sec:conjfalse}. 
In Section \ref{sec:geodesic} we define a rational map 
$\rho$ and its pullback map $\rho^*$, whose kernel is the ideal $P_G$.
We review properties of  block graphs including the existence of a unique shortest path. 
Using this uniqueness property, we define the ``shortest path map" $\psi$ and the initial term map $\phi$ and show that the two maps have the same kernel. 
We prove that the kernel of $\psi$ is equal to the ideal $CI_G$ for block graphs with one central vertex in Section \ref{sec:1central}. This result is generalized for all block graphs in Section \ref{sec:allgraphs}. 
Finally, in Section \ref{sec:SAGBI} we put all the pieces together to
prove Theorem \ref{thm: theConjecture}
using the results proved in the previous sections. We end the section by showing that the set $F$ forms a SAGBI basis ( Subalgebra Analog to Gr\"obner Basis for Ideals ) using the initial term map.


\section{Counterexamples to conjecture \ref{conj:false} }  \label{sec:conjfalse}

We first begin with some counterexamples to Conjecture \ref{conj:false}.  Initial
counterexamples suggest a modification of Conjecture \ref{conj:false} might be 
true, but we show that that strengthened version is also false.  This last
counterexample suggests that it is unlike that there is a repair for the conjecture.

\begin{ex}\label{ex:bigger}
Let $G=([6],E)$ be the graph as shown in Figure \ref{fig: counterexample 1}. 
Here $A=\{1,2\}, B=\{4,5,6\}$ and $C=\{3\}$. Computing the ideals $P_G$ and $P_{G_1}+P_{G_2}+ \langle 2\times 2 \text{ minors of } \Sigma_{A\cup C,B\cup C} \rangle$, we get   
\begin{eqnarray*}
P_G &=& \langle \sigma_{14}\sigma_{25}\sigma_{46}-
\sigma_{14}\sigma_{26}\sigma_{45}-
\sigma_{15}\sigma_{24}\sigma_{46}+
\sigma_{15}\sigma_{26}\sigma_{44}+
\sigma_{16}\sigma_{24}\sigma_{45}-
\sigma_{16}\sigma_{25}\sigma_{44},\\
&&\sigma_{24}\sigma_{45}\sigma_{56}-
\sigma_{24}\sigma_{46}\sigma_{55}-
\sigma_{25}\sigma_{44}\sigma_{56}+
\sigma_{25}\sigma_{46}\sigma_{45}+
\sigma_{26}\sigma_{44}\sigma_{55}-
\sigma_{26}\sigma_{45}^2 \rangle \\
&&+ P_{G_1}+P_{G_2} + \langle 2\times 2 \text{ minors of } \Sigma_{A\cup C,B\cup C} \rangle.
\end{eqnarray*}

\begin{figure}
\begin{center}
\begin{tikzpicture}
  [scale=1.5,auto=left,every node/.style={circle,fill=blue!20}]
\node (n1) at (0,1.5)  {1};
\node (n2) at (0,0) {2};
\node (n3) at (1,.75)  {3};
\node (n4) at (1.5,1.5)  {4};
\node (n5) at (1.5,0)  {5};
\node (n6) at (3,.75)  {6};

  \foreach \from/\to in {n1/n2,n1/n3,n2/n3,n3/n4,n3/n5,n4/n5,n4/n6,n5/n6}
    \draw (\from) -- (\to);
\end{tikzpicture}
\end{center}
\caption{\label{fig: counterexample 1} }
\end{figure}

\end{ex}

Note that even for some small block graphs Conjecture \ref{conj:false} is false.

\begin{ex}\label{ex:chain}
Consider the graph $G = ([4], E)$ which is a path of length $4$.  
Taking $c = \{3\}$, we get a decomposition of $G$ into $G_1$ and $G_2$ which are paths of 
length $3$ and $2$ respectively.  A quick calculation in Macaulay2 [\cite{Grayson(2017)}]
shows that $P_G = CI_G$ is generated by $5$ quadratic binomials.  However,
\[
P_{G_1} + P_{G_2} + 
\langle 2\times 2 \text{-minors of } \Sigma _{\{1,2, 3\}, \{3,4\} } \rangle
\]
has only $4$ minimal generators.
\end{ex}

Although $P_G$ is not equal to $P_{G_1}+P_{G_2}+\langle 2\times 2 \text{ minors of } 
\Sigma_{A\cup C,B\cup C} \rangle $ in these examples, we observe that the extra generators of $P_G$ are
also determinantal conditions arising from submatrices of $\Sigma$.  Furthermore,
they can be seen as being implied by the original rank conditions
in $P_{G_1}$ and $P_{G_2}$ plus the rank conditions that are implied by 
$\langle 2\times 2 \text{ minors of } 
\Sigma_{A\cup C,B\cup C} \rangle$.

For instance, in Example \ref{ex:chain}, the ideal $R_G =  P_{G_1} + P_{G_2} + 
\langle 2\times 2 \text{-minors of } \Sigma _{\{1,2, 3\}, \{3,4\} } \rangle$
is generated by the $2 \times 2$ minors of the two matrices
\[
\begin{pmatrix}
\sigma_{12} &   \sigma_{13} \\
\sigma_{22} & \sigma_{23}
\end{pmatrix}  \quad \quad \mbox{ and }\quad \quad
\begin{pmatrix}
\sigma_{13} &   \sigma_{14} \\
\sigma_{23} & \sigma_{24}  \\
\sigma_{33} & \sigma_{34}
\end{pmatrix}.
\]
Whereas the $P_G$ is generated by the $2 \times 2$ minors of the two matrices.
\[
\begin{pmatrix}
\sigma_{12} &   \sigma_{13} & \sigma_{14} \\
\sigma_{22} & \sigma_{23} & \sigma_{24}
\end{pmatrix}  \quad \quad \mbox{ and }\quad \quad
\begin{pmatrix}
\sigma_{13} &   \sigma_{14} \\
\sigma_{23} & \sigma_{24}  \\
\sigma_{33} & \sigma_{34}
\end{pmatrix}.
\]
However, we can take the generators $R_G$ and, assuming that $\sigma_{33}$ is not zero (which
is valid since $\Sigma$ is positive definite), 
we see that this implies that 
\[
\begin{pmatrix}
\sigma_{12} &   \sigma_{13} & \sigma_{14} \\
\sigma_{22} & \sigma_{23} & \sigma_{24}
\end{pmatrix}
\]
must be a rank $1$ matrix.  

Similarly, in Example \ref{ex:bigger}, 
we know that $(\{3\},\{6\},\{4,5\})$ is a separating partition 
for the subgraph $G_2$. So, the ideal $J_{\{3\}\indep \{6\}|\{4,5\}}$ 
is contained in $P_{G_2}$, which implies that rank of the submatrix 
$\Sigma_{\{3,4,5\},\{4,5,6\}}$ is $2$. Similarly, $(\{1,2\},\{4,5,6\},\{3\})$ 
is a separating partition of $G$, which implies that 
rank of the submatrix $\Sigma_{\{1,2,3\},\{3,4,5,6\}}$ is $1$. 
Now, as $\Sigma_{\{1,2,3\},\{4,5,6\}}$ is a submatrix of $\Sigma_{\{1,2,3\},\{3,4,5,6\}}$,
we can say that $\Sigma_{\{1,2,3\},\{4,5,6\}}$ also has rank $1$. Hence, from these two rank constraints 
and the added assumption that $\sigma_{33}$ is not zero we can conclude that the submatrix $\Sigma_{\{1,2,4,5\},\{4,5,6\}}$ has rank $2$.

The details of these examples suggest that a better version of the conjecture might be
\begin{eqnarray*}
P_G&=& {\rm Lift}(P_{G_1})+ {\rm Lift}(P_{G_2}) +\langle 2\times 2 \text{ minors of } \Sigma_{A\cup C,B\cup C} \rangle.
\end{eqnarray*}
Here ${\rm Lift}(P_{G_1})$ denotes some operation that takes the generators of $P_{G_1}$ 
and extends them to the whole graph, analogous to how the toric fiber product in \cite{Sullivant(2007)}
lifts generators for reducible hierarchical models on discrete variables [\cite{Dobra(2004),Hosten(2002)}].   
We do not make precise what this lifting operation could be, because if it preserves the degrees
of generating sets
the following example shows that
no precise version of this notion could make this conjecture be true.

\begin{ex}
Let $G=([7],E)$ be the graph as shown in Figure \ref{fig: counterexample 2} and let $(A,B,C)$ be the partition $(\{1,2,3\},\{5,6,7\},\{4\})$.
Computing the vanishing ideal, we get $P_G = CI_G$, but that among the
minimal generators of $P_G$ is one degree $4$ polynomial $m$
where 
\begin{eqnarray*}
m&=&\hspace{.5cm}\sigma_{17}^2\sigma_{23}\sigma_{56} -\sigma_{13}\sigma_{17}\sigma_{27}\sigma_{56}-\sigma_{12}\sigma_{17}\sigma_{37}\sigma_{56}+\sigma_{11}\sigma_{27}\sigma_{37}\sigma_{56} -\sigma_{16}\sigma_{17}\sigma_{23}\sigma_{57}
\\
&+&\sigma_{13}\sigma_{16}\sigma_{27}\sigma_{57}+\sigma_{12}\sigma_{16}\sigma_{37}\sigma_{57}-\sigma_{11}\sigma_{26}\sigma_{37}\sigma_{57}-\sigma_{15}\sigma_{17}\sigma_{23}\sigma_{67}+\sigma_{13}\sigma_{15}\sigma_{27}\sigma_{67}
\\
&+&\sigma_{12}\sigma_{15}\sigma_{37}\sigma_{67}-\sigma_{11}\sigma_{25}\sigma_{37}\sigma_{67}-\sigma_{12}\sigma_{13}\sigma_{57}\sigma_{67}+\sigma_{11}\sigma_{23}\sigma_{57}\sigma_{67}+\sigma_{15}\sigma_{16}\sigma_{23}\sigma_{77}
\\
&-&\sigma_{13}\sigma_{15}\sigma_{26}\sigma_{77}-\sigma_{12}\sigma_{15}\sigma_{36}\sigma_{77}+\sigma_{11}\sigma_{25}\sigma_{36}\sigma_{77}+\sigma_{12}\sigma_{13}\sigma_{56}\sigma_{77}-\sigma_{11}\sigma_{23}\sigma_{56}\sigma_{77}.
\end{eqnarray*}
As both $P_{G_1}$ and $P_{G_2}$ are generated by polynomials of degree $3$, 
this degree $4$ polynomial could not be obtained from a degree preserving lifting operation.

\begin{figure}
\begin{center}
\begin{tikzpicture}
  [scale=1.5,auto=left,every node/.style={circle,fill=blue!20}]
\node (n1) at (-1,.75) {1};
\node (n2) at (0,1.5)  {2};
\node (n3) at (0,0) {3};
\node (n4) at (1,.75)  {4};
\node (n5) at (2,1.5)  {5};
\node (n6) at (2,0)  {6};
\node (n7) at (3,.75)  {7};

  \foreach \from/\to in {n1/n2,n1/n3,n2/n4,n3/n4,n4/n5,n4/n6,n5/n7,n6/n7}
    \draw (\from) -- (\to);

\end{tikzpicture}
\end{center}
\caption{\label{fig: counterexample 2} }
\end{figure}
\end{ex}


\section{Shortest path in block graphs} \label{sec:geodesic}

Our goal for the rest of the paper is to prove Theorem \ref{thm: theConjecture}.
To do this, we need to phrase some parts in the language of commutative
algebra.  The vanishing ideal is the kernel of a certain ring homomorphism,
or the presentation ideal of a certain $\rr$-algebra.  
We will show that we can pass to a suitable initial algebra and analyze the
combinatorics of the resulting toric ideal.  This is proven in this section and
those that follow.

We begin this section by giving an overview of toric ideals. 
We then define a rational map $\rho$ such that 
the kernel of its pullback map gives us the ideal $P_G$. 
We also show the existence of a unique shortest path between any
two vertices of a block graph. 
This property allows us to define the ``shortest path map". 

Let $\mathcal{A}=\{a_1,a_2,...,a_n\}$ be a fixed subset of $\mathbb{Z}^d$. We consider the homomorphism
\begin{eqnarray*}
\pi: \mathbb{N}^n\rightarrow \mathbb{Z}^d, \hspace{1cm} u=(u_1,...,u_n)\mapsto u_1a_1+...+u_na_n.
\end{eqnarray*}
This map $\pi$ lifts to a homomorphism of subgroup algebras:
\begin{eqnarray*}
\hat{\pi}: \mathbb{R}[x_1,...,x_n]\rightarrow \mathbb{R}[t_1,...,t_d,t_1^{-1},...,t_d^{-1}], \hspace{.75cm}
x_i\mapsto t^{a_i}.
\end{eqnarray*}
The kernel of $\hat{\pi}$ is called the \textit{toric ideal} of $\mathcal{A}$. By Lemma 4.1 of \cite{Sturmfels(1996)} we know that the toric ideal can be generated by the set of binomials of the form 
\begin{eqnarray*}
\{x^u-x^v : u,v \in \mathbb{N}^n \text{ with } \pi(u)=\pi(v) \}.
\end{eqnarray*}
From the construction above we observe that any monomial map can be written as $\hat{\pi}$ for some given set of vectors $\mathcal{A}$. This gives us that the kernel of every monomial map is a toric ideal.

Now, let $\rr[K] = \rr[k_{11}, k_{12},...,k_{nn}]$ denote the polynomial
ring in the entries of the concentration matrix $K$, and $\rr(K)$ its
fraction field.

We define the rational map $\rho: \mathcal{L}\dashrightarrow{}\mathcal{L}^{-1}$ as follows:
\[
\rho(K)=\rho(k_{11},k_{12},\ldots,k_{nn})
=  
(\rho_{11}(k_{11},k_{12},\ldots,k_{nn}),\rho_{12}(k_{11},k_{12},\ldots,k_{nn}),...,\rho_{nn}(k_{11},k_{12},\ldots,k_{nn})),
\]
where $\rho_{ij}\in \mathbb{R}(K)$ is the 
$(i,j)$ coordinate of $K^{-1}$. The rational map does not yield a well defined function from $\mathcal{L}$ to $\mathcal{L}^{-1}$ as every matrix in $\mathcal{L}$ is not invertible (chapter 3, \cite{Hassett(2007)}).
Also note that the definition of $\rho$ depends on the underlying graph
$G$, since the zero pattern of $K$ is determined by $G$.

The \textit{pull-back} map of $\rho$ is
\[
\rho^*:\mathbb{R}[\Sigma] \rightarrow  \mathbb{R}(K), \quad 
\sigma_{ij} \mapsto  \rho_{ij}(K). 
\]
So, for each $p\in \mathbb{R}[\Sigma]$ and $K\in \mathcal{L}$,
\[\rho^*(p)(K)=p\circ \rho(K)
=p(\rho_{11}(K),\rho_{12}(K),...,\rho_{nn}(K)).
\]
Hence, we have
\[
P_G=\mathcal{I(L}^{-1})= \ker(\rho^*).
\]

For a given graph $G=([n],E)$, 
let $f_{ij}\in \mathbb{R}[K]$ be the polynomial defined as 
$\det(K)$ times the $(i,j)$ coordinate of the matrix $K^{-1}$. 
Let $F=\{f_{ij}:1\leq i\leq j\leq n\}$. So, the map $\rho^*$ can be written as
\[
\rho^*:\mathbb{R}[\Sigma] \rightarrow  \mathbb{R}(K) \,  \quad 
\rho^*(\sigma_{ij}) = \frac{1}{\det(K)}\cdot f_{ij}.
\]
As $1/\det(K)$ is a constant which is present in the image of every $\sigma_{ij}$, 
removing that factor from every image would not change the kernel of 
$\rho^*$. Hence, we change the map $\rho^*$ as 
\[
\rho^*:\mathbb{R}[\Sigma] \rightarrow \mathbb{R}[F], \quad  \rho^*(\sigma_{ij})= f_{ij},
\]
where $\mathbb{R}[F]=\mathbb{R}[f_{11},f_{12},...,f_{nn}]\subseteq \mathbb{R}[K]$.

\begin{ex} \label{ex:illustrate}
Let $G=([4],E)$ be a graph with 4 vertices as shown in Fig \ref{fig: firstexample}. The matrices $\Sigma$ and $K$ for this graph are:
\[
\Sigma=
\begin{bmatrix}
    \sigma_{11}       & \sigma_{12} & \sigma_{13} & \sigma_{14} \\
    \sigma_{12}       & \sigma_{22} & \sigma_{23} & \sigma_{24} \\
    \sigma_{13}       & \sigma_{23} & \sigma_{33} & \sigma_{34} \\
    \sigma_{14}       & \sigma_{24} & \sigma_{34} & \sigma_{44}
    \end{bmatrix},  \quad \quad 
K=
\begin{bmatrix}
    k_{11}       & k_{12} & k_{13} & 0 \\
    k_{12}       & k_{22} & k_{23} & 0 \\
    k_{13}       & k_{23} & k_{33} & k_{34} \\ 0       & 0 & k_{34} & k_{44}
    \end{bmatrix}.
\]

The ideal $P_G$ can be calculated by using the equation $\Sigma\cdot K=Id_4$ and eliminating
the $K$ variables.  

\begin{eqnarray*}
\langle\Sigma\cdot K-Id_4\rangle & =  & 
\langle
\sigma_{11}k_{11}+\sigma_{12}k_{12}+\sigma_{13}k_{13}-1, \sigma_{11}k_{12}+\sigma_{12}k_{22}+\sigma_{13}k_{23},
\ldots,  \\
&   & 
\quad \quad  \sigma_{14}k_{13}+\sigma_{24}k_{23}+\sigma_{34}k_{33}+\sigma_{44}k_{34},\sigma_{34}k_{33}+\sigma_{44}k_{44}-1
\rangle.
\end{eqnarray*}

Eliminating the $K$ variables, we get 
\[
P_G \, \, = \, \,  \langle\Sigma\cdot K-Id_4 \rangle  \cap  \rr[\Sigma] \, \,  = \, \, 
 \langle \sigma_{13}\sigma_{34}-\sigma_{14}\sigma_{33} ,
 \sigma_{23}\sigma_{34}-\sigma_{24}\sigma_{33} ,
 \sigma_{14}\sigma_{23}-\sigma_{13}\sigma_{24} \rangle.
\]

From the map $\rho^*$, we have 
\begin{eqnarray}  \label{eqn:fs}
\begin{split}
 f_{11} &=  \underline{k_{22}k_{33}k_{44}}-k_{22}k_{34}^2-k_{23}^2k_{44} \\
 f_{22} &=  \underline{k_{11}k_{33}k_{44}}-k_{11}k_{34}^2-k_{13}^2k_{44} \\
 f_{33} &=  \underline{k_{11}k_{22}k_{44}}-k_{44}k_{12}^2 \\
 f_{44} &= \underline{k_{11}k_{22}k_{33}}-k_{11}k_{23}^2-k_{12}^2k_{33} \\  & +k_{12}k_{13}k_{23}+k_{13}k_{12}k_{23}-k_{13}^2k_{22} \\
  \end{split}
  \quad \quad
 \begin{split}
  f_{12} &= -\underline{ k_{12}k_{33}k_{44}}-k_{12}k_{34}^2-k_{23}k_{13}k_{44} \\
 f_{13} &=  -\underline{k_{13}k_{22}k_{44}}+k_{12}k_{23}k_{44} \\
 f_{14} &=  \, \, \, \,  \underline{k_{13}k_{34}k_{22}} - k_{12}k_{23}k_{34} \\
 f_{23} &=  -\underline{k_{23}k_{11}k_{44}} + k_{12}k_{13}k_{44} \\
 f_{24} &= \, \, \, \,  \underline{k_{23}k_{34}k_{11}}-k_{34}k_{13}k_{12} \\
 f_{34} &= -\underline{ k_{34}k_{11}k_{22}}+k_{34}k_{12}^2 \\
 \end{split}
\end{eqnarray}
 where $f_{ij}$ is $\det(K)$ times the $(i,j)$ coordinate of $K^{-1}$.
Evaluating the kernel of $\rho^*$, we get
\[
\ker(\rho^*) \, \,= \, \, \langle \sigma_{13}\sigma_{34}-\sigma_{14}\sigma_{33} ,
\sigma_{23}\sigma_{34}-\sigma_{24}\sigma_{33} ,\sigma_{14}\sigma_{23}-\sigma_{13}\sigma_{24} \rangle
\]
which is same as the ideal $P_G$. Note that $G$ is a block graph with a 
single $1$-clique sum decomposition. As the generators of $P_G$ are the 
$2\times 2$ minors of $\Sigma_{\{1,2,3\}, \{3,4\} }$, 
the conjecture holds for this example.

\begin{figure}
\begin{center}
\begin{tikzpicture}
  [scale=.6,auto=left,every node/.style={circle,fill=blue!20}]
  \node (n4) at (9,0)  {4};
  \node (n1) at (0,0) {1};
  \node (n2) at (3,3)  {2};
  \node (n3) at (5.5,0)  {3};

  \foreach \from/\to in {n1/n2,n2/n3,n3/n4,n1/n3}
    \draw (\from) -- (\to);

\end{tikzpicture}
\end{center}
\caption{\label{fig: firstexample} }
\end{figure}
\end{ex}

Observe that in Example \ref{ex:illustrate}, 
each $f_{ij}$ contains a monomial which corresponds to the shortest 
path from $i$ to $j$ in the graph $G$ along with loops at the vertices 
not in the path. For example, $f_{24}$ has the monomial 
$k_{23}k_{34}k_{11}$ where $k_{23}k_{34}$ corresponds to the 
shortest path from $2$ to $4$ and $k_{11}$ corresponds to the 
loop at the vertex $1$.  In the (\ref{eqn:fs}), the underlined
term is this special term.

This turns out to be important  in our proofs, and we 
formalize this observation in Proposition \ref{prop:shortestpath}.
We now look at some properties of block graphs and 1-clique partitions 
in order to prove the existence of shortest paths.

\begin{prop}\label{prop:uniqueness}
If $G$ is a block graph, then for any two vertices $i$ and $j$ there exists 
a unique shortest path in $G$ connecting them. 
Further, if $(A,B,C)$ is a 1-clique partition of $G$ with 
$c\in C$ and if $i\in A$ and $j\in B$, then the unique shortest path from $i$ to $j$ can be decomposed 
into the unique shortest paths from $i$ to $c$ and $c$ to $j$.
\end{prop}

\begin{proof}
We prove this by applying induction on the number of vertices in $G$. 
If $i$ and $j$ are connected by a single edge, 
then that is the unique shortest path. 
If they are not connected by a single edge, 
then there exists a 1-clique partition $(A,B,C)$ with $C = \{c\}$ which separates them. 
But as $A\cup C$ and $B\cup C$ are also block graphs and have fewer
vertices than $G$, by induction there exist unique shortest paths from 
$i$ to $c$ and from $c$ to $j$. But as any path from $i$ to $j$ must pass through 
$c$, the concatenation of the unique shortest paths from $i$ to $c$ and 
$c$ to $j$ would be the unique shortest path from $i$ to $j$. 

The second part follows from a property of unique shortest paths that 
if $c$ is a point on the path, then the subpaths from $i$ to $c$ and 
$c$ to $j$ are the unique shortest paths from $i$ to $c$ and $c$ to $j$ respectively.
\end{proof}

For the rest of the paper, we assume that $G$ is a block graph and
the shortest path from $i$ to $j$  in $G$ is denoted by $i\leftrightarrow j$.
We use $(i', j') \in i \leftrightarrow j$ to indicate that the edge
$(i', j')$ appears in the path $i\leftrightarrow j$.  We let $\ell(i,j)$ denote
the length of the shortest path from $i$ to $j$.
We now state a result from \cite{Jones(2005)} which 
will be used to prove Proposition \ref{prop:shortestpath}.

\begin{thm}\label{thm:Covariance} (Theorem 1, \cite{Jones(2005)}) 
Consider an $n-$dimensional multivariate normal distribution with a finite and 
non-singular covariance matrix $\Sigma$, with precision matrix $K=\Sigma^{-1}$. 
Let $K$ determine the incidence matrix of a finite, 
undirected graph on vertices $\{1,...,n\}$, with nonzero elements in $K$ corresponding to edges. 
The element of $K$ corresponding to the covariance between variables $x$ and $y$ can be written as a sum of path weights over all paths in the graph between $x$ and $y$:
\[
\sigma_{xy}=\sum_{P\in \mathscr{P}_{xy}}(-1)^{m+1}k_{p_1p_2}k_{p_2p_3}...k_{p_{m-1}p_m}\frac{\det(K_{\setminus P})}{\det(K)},
\]
where $\mathscr{P}_{xy}$ represents the set of paths between $x$ and $y$, so that $p_1=x$ and $p_m=y$ for all $P\in \mathscr{P}_{xy}$ and $K_{\setminus P}$ is the matrix with rows and columns corresponding to variables in the path $P$ omitted, with the determinant of a zero-dimensional matrix taken to be 1.
\end{thm}

\begin{prop}\label{prop:shortestpath}
Let $G=([n],E)$ be a block graph with the corresponding concentration matrix $K$. 
If $f_{xy}$ denote $\det(K)$ times the $(x,y)$ coordinate of $K^{-1}$, then $f_{xy}$ has the monomial 
\[
(-1)^{\ell(i,j) } \prod_{(x',y')\in x\leftrightarrow y}k_{x'y'}\prod_{t\notin x\leftrightarrow y}k_{tt}
\]
as one of its terms. Furthermore, this term has the highest number of diagonal entries $k_{tt}$ among all the monomials of $f_{xy}$.
\end{prop}

\begin{proof}
From Theorem \ref{thm:Covariance}, we have 
\[
f_{xy}=\det(K)\cdot\sigma_{xy}= \sum_{P\in \mathscr{P}_{xy}}(-1)^{m+1}k_{p_1p_2}k_{p_2p_3}...k_{p_{m-1}p_m}\text{det}(K_{\setminus P}).
\]

From Proposition \ref{prop:uniqueness} we know that if $G$ is a block graph, 
then for any two vertices $x$ and $y$, there exists a unique shortest path 
between $x$ and $y$. If $z\in x\leftrightarrow y$ with $z\neq x,y$, 
then there exists a 1-clique partition $(A,B,C)$ of $G$ with $C=\{z\}$ and 
$x\in A, y\in B$. By the definition of 1-clique partition we know that any 
path from $x$ to $y$ must pass through $z$. As $z$ is arbitrarily chosen, 
any path in $G$ from $x$ to $y$ must pass through all the vertices in 
$x\leftrightarrow y$. This gives us that the unique shortest path has the 
least number of vertices among all the other paths from $x$ to $y$. 
So, the matrix $K_{\setminus x\leftrightarrow y}$ has the highest dimension 
among all the other matrices $K_{\setminus P}, P\in \mathscr{P}_{xy}$.

Now, for any $P\in \mathscr{P}_{xy}$, det($K_{\setminus P}$) contains the
monomial $\prod_{t\notin P}k_{tt}$ as $G$ is assumed to have self loops. 
This monomial has the highest number of diagonals among all the monomials 
in $\det(K_{\setminus P}$) as the degree of $\det(K_{\setminus P}$) is same as the degree of $\prod_{t\notin P}k_{tt}$. So, the monomial
\[\prod_{(x',y')\in P}k_{x'y'}\prod_{t\notin P}k_{tt}
\]
has the highest number of diagonal terms among all the monomials in $\prod_{(x',y')\in P}k_{x'y'}\text{det}(K_{\setminus P})$. As $K_{\setminus x\leftrightarrow y}$ has the highest dimension, we can conclude that the monomial 
\[\prod_{(x',y')\in x\leftrightarrow y}k_{x'y'}\prod_{t\notin x\leftrightarrow y}k_{tt}
\]
has the maximum number of diagonal terms among all the monomials in $f_{xy}$.
\end{proof}

We call the monomial defined above as the \emph{shortest path monomial} of $f_{ij}$.
As the shortest path monomial in each $f_{ij}$ has the highest power of diagonals $k_{tt}$
among all the other monomials in $f_{ij}$, we can define a weight order on $\mathbb{R}[K]$ 
where the weight of any monomial is the number of diagonal entries 
of the monomial. The initial term of $f_{ij}$ in this order will 
be precisely the shortest path
monomial.

\begin{defn}\label{def:initialmonomialmap}
Let $G$ be a block graph.  Define the $\rr$-algebra homomorphism
\[
\phi:  \rr[\Sigma] \rightarrow \rr[K], \quad  \sigma_{ij}  \mapsto   
\prod_{(i',j')\in i\leftrightarrow j}k_{i'j'}\prod_{t\notin i\leftrightarrow j}k_{tt}.
\]
This monomial homomorphism is called the \emph{initial term map}.
\end{defn}

The  map $\phi$ is the initial term map of $\rho^*$, but with the sign  $(-1)^{\ell(i,j)}$ omitted. 
We will use this to show that the set $F$ forms a 
SAGBI basis of $\mathbb{R}[F]$ by using this term order, as part of our proof of Theorem 
\ref{thm: theConjecture}.  This appears in Section \ref{sec:SAGBI}.  To do this we must spend some time proving properties of $\phi$
and $\ker \phi$.  

Note that the kernel of $\phi$ is the same with or without the signs $(-1)^{\ell(i,j)}$.  This is because
the monomials that appear are graded by the number of diagonal terms that appear, which is
also counted by the  $(-1)^{\ell(i,j)}$.  Any binomial relation $\sigma^u - \sigma^v \in \ker \phi$
much also lead to the same power of negative one on both sides of the equation.

From the standpoint of proving results about this monomial map based on shortest
paths in a block graph, it turns out to be easier to work with a related map
that we call the shortest path map.

\begin{defn}\label{defn:psi map}
Let $G=([n],E)$ be a block graph. The \emph{shortest path map} $\psi$ is defined as
\begin{eqnarray*}
\psi&:&\mathbb{R}[\Sigma]\rightarrow \mathbb{R}[a_1,...,a_n,k_{12},...,k_{n-1,n}]  =  \rr[A, K] \\
\psi(\sigma_{ij})&=& \left\{
\begin{array}{ll}
      a_ia_j\prod_{(i',j') \in i\leftrightarrow j} k_{i'j'} & i\neq j \\
      a_{i}^2 & i= j. \\
\end{array} 
\right.
\end{eqnarray*}
\end{defn}

\begin{ex}\label{ex:Mphi}
Let $G$ be the graph in Example \ref{ex:illustrate}. 
Let $\psi$ be the shortest path map and $\phi$ the initial monomial map
as given in Definitions \ref{def:initialmonomialmap} and \ref{defn:psi map}.  
So for example,
\[
\phi(\sigma_{11}) = k_{22} k_{33} k_{44}, \phi(\sigma_{12}) = k_{12} k_{33}k_{44}, \ldots
\]
\[
\psi(\sigma_{11}) =  a_1^2,  \psi(\sigma_{12}) = a_1 a_2 k_{12}, \ldots.
\]

As is typical for monomial parametrizations, we can represent them by matrices whose 
columns are the exponent vectors of the monomials appearing in the parametrization.
In this case, we get the following matrices corresponding to $\phi$ and $\psi$ respectively.
\[  M_\phi = \begin{bmatrix}
    0 & 0 & 0 & 0 & 1 & 1 & 1 & 1 & 1 & 1 \\
    1 & 0 & 1 & 1 & 0 & 0 & 0 & 1 & 1 & 1 \\
    1 & 1 & 0 & 0 & 1 & 0 & 0 & 0 & 0 & 1 \\
    1 & 1 & 1 & 0 & 1 & 1 & 0 & 1 & 0 & 0 \\
    0 & 1 & 0 & 0 & 0 & 0 & 0 & 0 & 0 & 0 \\
    0 & 0 & 1 & 1 & 0 & 0 & 0 & 0 & 0 & 0 \\
    0 & 0 & 0 & 0 & 0 & 1 & 1 & 0 & 0 & 0 \\
    0 & 0 & 0 & 1 & 0 & 0 & 1 & 0 & 1 & 0  
\end{bmatrix}  \quad \quad \quad \quad M_\psi = \begin{bmatrix}
    2 & 1 & 1 & 1 & 0 & 0 & 0 & 0 & 0 & 0 \\
    0 & 1 & 0 & 0 & 2 & 1 & 1 & 0 & 0 & 0 \\
    0 & 0 & 1 & 0 & 0 & 1 & 0 & 2 & 1 & 0 \\
    0 & 0 & 0 & 1 & 0 & 0 & 1 & 0 & 1 & 2 \\
    0 & 1 & 0 & 0 & 0 & 0 & 0 & 0 & 0 & 0 \\
    0 & 0 & 1 & 1 & 0 & 0 & 0 & 0 & 0 & 0 \\
    0 & 0 & 0 & 0 & 0 & 1 & 1 & 0 & 0 & 0 \\
    0 & 0 & 0 & 1 & 0 & 0 & 1 & 0 & 1 & 0  
\end{bmatrix}.
\]
The rows of $M_\phi$ are ordered as 
$\{k_{11},k_{22},k_{33},k_{44},k_{12},k_{13},k_{23},k_{34}\}$ 
and the rows of $M_\psi$ are ordered as $\{a_1,a_2,a_3,a_4,k_{12},k_{13},k_{23},k_{34}\}$. 
\end{ex}

In fact, these two monomial maps have the same kernel for block graphs. 

\begin{prop}\label{prop:same kernel}
Let $G$ be a block graph and let  $\phi$ and $\psi$ be the initial
term map and the shortest path map, respectively. Then $\ker(\psi) = \ker(\phi)$.
\end{prop}

\begin{proof}
Both $\ker(\phi)$ and $\ker(\psi)$ are toric ideals.  To show that they have the same
kernel, it suffices to show that the associated matrices of exponent vectors have the same kernel,
or equivalently, that they have the same row span.  
Let $M_\phi$ and $M_\psi$ denote those matrices.
As $\psi(\sigma_{ij})=a_ia_j\prod_{(i',j')\in i\leftrightarrow j}k_{i'j'}$ and
$ \phi(\sigma_{ij})=\prod_{(i',j')\in i\leftrightarrow j}k_{i'j'}\prod_{s\notin i\leftrightarrow j}k_{ss}$, 
the rows corresponding to $k_{ij}$ with $i\neq j$ remain the same in both the matrices. 
So, we only need to write the $k_{ii}$ rows of $M_\phi$ as a linear combination of the rows of $M_\psi$
and vice versa.

The row vector corresponding to $k_{ii}$ in $M_\phi$ is 1 at the 
$\sigma_{pq}$ coordinates where $i\notin p\leftrightarrow q$ and is 0 elsewhere. 
Similarly, the row vector corresponding to $a_i$ in $M_\psi$ is 2 at the $\sigma_{ii}$ coordinate, 
1 at the $\sigma_{pq}$ coordinates where either of the end  
points is $i$ (either $p=i$ or $q=i$) and 0 elsewhere.

We observe that the $k_{ii}$ rows of $M_\phi$ can be written as a 
linear combination of the rows of $M_\psi$ using the following relation:  
\begin{eqnarray}\label{eqnarray:kii}
2k_{ii}=\sum_{j\neq i}a_j-\sum_{s:i\leftrightarrow s \text{ is an edge}}k_{is}.
\end{eqnarray}
Here we are using $k_{ii}$ to denote the row vector of $M_\phi$ 
corresponding to the indeterminate $k_{ii}$, and similarly for $a_j$ and $k_{is}$. 
We have
\begin{eqnarray*}
\sum_{j\neq i}a_j &=& \text{ paths ending at } i + 2(\text{ paths not ending at } i) - i\leftrightarrow i ,\\
\sum_{s: i\leftrightarrow s \text{ is an edge }}k_{is}&=& \text{ paths ending at } i +2(\text{ paths containing } i \text{ but not ending at } i)-i\leftrightarrow i .
\end{eqnarray*}
So,
\[
\sum_{j\neq i}a_j-\sum_{s: i\leftrightarrow s \text{ is an edge }}k_{is}= 2(\text{ paths not containing } i)= 2k_{ii}.
\]
As this relation is true for any $i$, the row space of $M_\phi$ is contained in the row space of $M_\psi$.
So, $\ker(\psi)\subseteq \ker(\phi)$.

To get the reverse containment, we need to write the $a_i$ rows of $M_\psi$ 
as a linear combination of the rows of $M_\phi$.
From (\ref{eqnarray:kii}), we get 
\[
\sum_{j\neq i}a_j=2k_{ii}+\sum_{s:i\leftrightarrow s \text{ is an edge}}k_{is}.
\]
Writing these $n$ equations in the matrix form, 
we get an $n\times n$ matrix in the left hand side 
which has $0$ in its diagonal entries and $1$ elsewhere.
As this matrix is invertible for any $n>1$, we can conclude that the row space of 
$M_\psi$ is contained in the row space of $A$. Hence, $\ker(\psi)=\ker(\phi)$.
\end{proof}

Our goal in the next two sections will be to characterize the vanishing ideal of
the shortest path map for block graphs.

\begin{defn}
Let $G$ be a block graph.
Let $SP_G =  \ker(\psi) = \ker(\phi)$ be the kernel of the
shortest path map.  This ideal is called the \emph{shortest path ideal}.  
\end{defn}

As the shortest path map is a monomial map, we know that the shortest path ideal is a toric ideal. We will eventually show that $SP_G = CI_G = P_G$, however we find it
useful to have different notation for these ideals while we have not
yet proven the equality.


\section{Shortest path map for block graphs with 1 central vertex} \label{sec:1central}

In this section we show that $SP_G =  CI_G$ in the case that $G$
is a block graph with only one central vertex.  This will be an important special
case and tool for proving that $SP_G =  CI_G$ for all block graphs,
which we do in Section \ref{sec:allgraphs}.  Our proof for 
graphs with only one central vertex depends on reducing the study of the
ideal $SP_G$ in this case to related notions of edge rings in \cite{DeLoera(1995)} and \cite{Herzog(2018)}.

\begin{defn}\label{defn:central vertex}
If $G$ is a block graph, a vertex $c$ in $G$ is called a 
\textit{central vertex} if there exists a $1$-clique partition $(A,B,C)$ of $G$ 
such that $C=\{c\}$.
\end{defn}

\begin{ex}\label{ex:1centralclique}
Let $G$ be the block graph with 5 vertices as in Figure \ref{fig: 1central}. 
There are three possible 1-clique partitions of $G$, 
$(\{1,2\},\{4,5\},\{3\}),(\{1,2,4\},\{5\},\{3\})\text{ and }(\{1,2,5\},\{4\},\{3\})$. 
We see that $3$ is the only central vertex of $G$ as $C=\{3\}$ for all the three partitions. 
Now computing $SP_G$ for this graph, we get 
\begin{eqnarray*}
\ker(\psi)=\langle \sigma_{34}\sigma_{35}-\sigma_{33}\sigma_{45}, \sigma_{24}\sigma_{35}-\sigma_{23}\sigma_{45},
\sigma_{14}\sigma_{35}-\sigma_{13}\sigma_{45}, \sigma_{25}\sigma_{34}-\sigma_{23}\sigma_{45},
\\
\sigma_{15}\sigma_{34}-\sigma_{13}\sigma_{45},\sigma_{25}\sigma_{33}-\sigma_{23}\sigma_{35},
\sigma_{24}\sigma_{33}-\sigma_{23}\sigma_{34},
\sigma_{15}\sigma_{33}-\sigma_{13}\sigma_{35},\\
\sigma_{14}\sigma_{33}-\sigma_{13}\sigma_{34}, \sigma_{15}\sigma_{24}-\sigma_{14}\sigma_{25},\sigma_{15}\sigma_{23}-\sigma_{13}\sigma_{25}, \sigma_{14}\sigma_{23}-\sigma_{13}\sigma_{24}\rangle.
\end{eqnarray*}

\begin{figure}
\begin{center}
\begin{tikzpicture}
  [scale=.6,auto=left,every node/.style={circle,fill=blue!20}]
  \node (n2) at (0,0) {2};
  \node (n3) at (3,3)  {3};
  \node (n5) at (5.5,0)  {5};
  \node (n4) at (5.5,6)  {4};
  \node (n1) at (0,6)  {1};

  \foreach \from/\to in {n1/n2,n2/n3,n3/n4,n3/n5,n1/n3}
    \draw (\from) -- (\to);

\end{tikzpicture}
\end{center}
\caption{\label{fig: 1central} }
\end{figure}
\end{ex}  

We observe that in Example \ref{ex:1centralclique}, none of the generators of $SP_G$ 
contain the terms $\sigma_{12},\sigma_{11},\sigma_{22}, \sigma_{44}$ and $\sigma_{55}$. 
These terms correspond to the edges in $G$ which cannot be separated by 
any 1-clique partition of $G$. This property is true for all block graphs with one central vertex as we prove it in the next Lemma.

\begin{lem}\label{lem:inseparable edges}
Let $G$ be a block graph with one central vertex $c$ 
and let $D$ be the set of variables $\sigma_{pq}$, where 
the shortest path $p \leftrightarrow q$ does not intersect $c$.  
Then none of the variables appearing in $D$ appear in any of the minimal 
generators of the kernel of $\psi$.
\end{lem}

\begin{proof}
Since $\psi$ is a monomial parametrization, the
 kernel of $\psi$ is a homogeneous binomial ideal. Let 
\[
f= \sigma^{u}-\sigma^{v}
\]
be an arbitrary binomial in any generating set for the kernel of $SP_G$.  
 In particular, this implies
that $\sigma^{u}$ and $\sigma^{v}$ have no common factors.
Suppose by way of contradiction that
$\sigma_{pq}$ is  some variable in $D$ that divides one of the terms of $f$, 
say $\sigma^{u}$. Then $\psi(\sigma^{u})$ would have $k_{pq}$ as a factor. 
But $k_{pq}$ appears only in the image of $\sigma_{pq}$ as no other shortest path between 
any two vertices in $G$ contains the edge $(p,q)$. This would imply that 
$\sigma_{pq}$ is also a factor of $\sigma^{v}$ contradicting the fact that 
$\sigma^{u}$ and $\sigma^{v}$ have no common factors. 

Similarly, if $\sigma_{pp}$ is a factor of $\sigma^{u}$ where $p$ is not the central vertex, 
then $\psi(\sigma^{u})$ would have $a_p^2$ as a factor. 
In order to have $a_p^2$ as a factor of $\psi(\sigma^{v})$, it would require two variables in
$\sigma^v$
to have $p$ as one of their end points. 
As $p$ is not a central vertex, we will have 
$k_{cp}^2$ as a factor of $\psi(\sigma^{v})$.  But then this means that there must be two
variables in $\sigma^{u}$ that touch vertex $p$.  Which in turn forces another factor
of $a_p^2$ to divide $\psi(\sigma^{u})$.  Which in turn forces another two variables in $\sigma^v$ to
touch vertex $p$, and so on.  This process never terminates, showing that it is impossible that 
$\sigma_{pp}$ is a factor of $\sigma^{u}$.

Hence we can conclude that none of the variables in $D$ appear in any of the generators of $SP_G$. 
\end{proof}

Note that the proof of Lemma \ref{lem:inseparable edges} also applies to any block graph 
with multiple central vertices.  Hence, we can eliminate some of the variables 
in the computation of the shortest path ideal. 

We let $\rr[\Sigma \setminus D]$ denote the polynomial ring with the 
variables $D$ eliminated.  Here we are always taking $D$ to the be set
of variables corresponding to paths that do not touch the central vertex $x$.
Lemma \ref{lem:inseparable edges}
shows that it suffices to consider the problem of finding a generating set
of $SP_G$ inside of $\rr[\Sigma \setminus D]$.

The next step in our analysis of $SP_G$ for block graphs with one
central vertex will be to relate this ideal to a simplified parametrization
which we can then relate to edge ideals.

Let $G$ be a block graph with one central vertex.
Consider the map
\[
\hat{\psi} :  \rr[\Sigma \setminus D]  \rightarrow  \rr[a], \quad  \sigma_{ij}  \mapsto a_i a_j.
\]

\begin{prop}
Let $G$ be a block graph with one central vertex.  Then $\ker \hat{\psi} = \ker \psi$.  
\end{prop}

\begin{proof}
Note that because we only consider $\sigma_{pq} \in \rr[\Sigma \setminus D] $ then
any time $\psi(\sigma_{pq})$ contains $k_{pc}$ it will automatically contain
$a_p$ as well, and vice versa.  Hence, the $a_p k_{pc}$ always occurs as a factor together
in $\psi(\sigma_{pq})$.  So we can eliminate the $k_{pc}$ from the parametrization without
affecting the kernel of the homomorphism.
\end{proof}

In order to analyze $SP_G = \ker \hat{\psi} = \ker \psi$, we find it useful to first extend
the map to all of $\rr[\Sigma]$, where the kernel is well understood. 
 In particular,
we associate an edge in the graph $K_n^\circ$ to each variable in $\rr[\Sigma]$,
where $K_n^\circ$ denotes the complete graph $K_n$ with a loop added to each vertex.
We embed $K_n^\circ$ in the plane so that the vertices are arranged to lie on a circle.
We consider the map 
\[
\hat{\psi}:  \rr[\Sigma]  \rightarrow \rr[a], \quad  \sigma_{ij} =  a_i a_j
\]
and its kernel $SP_{K_n^\circ} =  \ker \hat{\psi}$.    
We describe a Gr\"obner basis for this ideal, based on the combinatorics of 
the embedding of the graph $K_n^\circ$.  We consider a pair of edges $(i,j), (k,l)$
to be \emph{intersecting} if the two edges share a vertex or the edges
intersect each other in the circular embedding of $K_n^\circ$.

The \emph{circular distance} between two vertices of $K_n$ 
is defined as the length of the shorter path among the two paths
present along the edges of the $n$-gon. We define the \emph{weight}
of the variable $\sigma_{ij}$ as the number of edges of $K_n^\circ$ that do not 
intersect the edge $(i,j)$. Let $\prec$ denote any term order that 
refines the partial order on monomials specified by these weights. 
Now, for any pair of non-intersecting edges $(i,j),(k,l)$ of $K_n^\circ$, 
one of the pairs $(i,k),(j,l)$ or $(i,l)(j,k)$ is intersecting. 
If $(i,k),(j,l)$ is the intersecting pair, we associate the binomial 
$\sigma_{ij}\sigma_{kl}-\sigma_{ik}\sigma_{jl}$ with the non intersecting pair of 
edges $(i,j),(k,l)$. We denote by $S'$ the set of all binomials obtained in this way.

\begin{lem}\label{lem:disjoint}
For any binomial $\sigma_{ij}\sigma_{kl}-\sigma_{ik}\sigma_{jl}$, 
where $(i,j),(k,l)$ are non-intersecting edges and $(i,k),(j,l)$ intersect, 
the initial term with respect to $\prec$ corresponds to the non intersecting edges 
in $K_n^\circ$.
\end{lem}

\begin{proof}
We divide the set of vertices in $K_n^\circ$ into four different parts (excluding the vertices $i,j,k$ and $l$). Let $P_1$ denote the set of vertices that are present in the path between $i$ and $j$ along the edges of the $n$-gon that do not contain $k$ and $l$. Similarly, let $P_2,P_3$ and $P_4$ denote the set of vertices between $j$ and $k$, $k$ and $l$ and $l$ and $i$ respectively. Let the cardinality of each $P_i$ be $p_i$ for $i=1,2,3,4$. Then, the weight of the four variables are as follows:
\begin{eqnarray*}
w(\sigma_{ij})&=&\sum_{i=1}^4{p_i\choose2}+p_2p_3+p_2p_4+p_3p_4+2(p_2+p_3+p_4)+1+(n-2) \\
w(\sigma_{kl})&=&\sum_{i=1}^4{p_i\choose2}+p_1p_2+p_1p_4+p_2p_4+2(p_1+p_2+p_4)+1+(n-2) \\
w(\sigma_{ik})&=&\sum_{i=1}^4{p_i\choose2}+p_1p_2+p_3p_4+p_1+p_2+p_3+p_4+(n-2) \\
w(\sigma_{jl})&=&\sum_{i=1}^4{p_i\choose2}+p_1p_4+p_2p_3+p_1+p_2+p_3+p_4+(n-2).
\end{eqnarray*}
This gives us  
\begin{eqnarray*}
w(\sigma_{ij})+w(\sigma_{kl})-(w(\sigma_{ik})+w(\sigma_{jl}))=2p_2p_4+2(p_2+p_4)+2>0.
\end{eqnarray*}
Hence, the initial term of $\sigma_{ij}\sigma_{kl}-\sigma_{ik}\sigma_{jl}$ with respect to $\prec$ is $\sigma_{ij}\sigma_{kl}$. Further, if $k=l$ then we have the binomial $\sigma_{ij}\sigma_{kk}-\sigma_{ik}\sigma_{jk}$ where 
\begin{eqnarray*}
w(\sigma_{kk})&=& {n-1\choose 2} + n-1 \text{  and } \\
w(\sigma_{jk})&=&\sum_{i=1}^4{p_i\choose 2}+p_1p_4+p_1p_3+p_3p_4+2(p_1+p_3+p_4)+1+(n-2).
\end{eqnarray*}
This gives us
\begin{eqnarray*}
w(\sigma_{ij})+w(\sigma_{kk})-(w(\sigma_{ik}+w(\sigma_{jk}))&=&\sum_{i=2}^4 \frac{p_i}{2}+2(p_2p_3+p_2p_4)+\frac{3}{2}(p_2+p_3+p_4)+p_2+4 > 0.
\end{eqnarray*}
So, the initial term of $\sigma_{ij}\sigma_{kk}-\sigma_{ik}\sigma_{jk}$ with respect to $\prec$ is $\sigma_{ij}\sigma_{kk}$.
\end{proof}

\begin{lem}\label{lem:added loops}
Let $S'$ be the set of binomials obtained from all the 
pairs of non-intersecting edges of $K_n^\circ$. 
Then $S'$ is the reduced Gr{\"o}bner basis of $SP_{K_n^\circ}$ with respect to $\prec$.
\end{lem}

\begin{proof}
By lemma \ref{lem:disjoint} we know that for any binomial $\sigma_{ij}\sigma_{kl}-\sigma_{il}\sigma_{jk}\in S'$, 
where $(i,j),(k,l)$ are non-intersecting edges and $(i,l),(j,k)$ intersect, 
the initial term with respect to $\prec$ corresponds to the non intersecting 
edges in $K_n^\circ$. Clearly, $\sigma_{ij}\sigma_{kl}-\sigma_{il}\sigma_{jk} \in SP_{K_n^\circ}$. 

The proof follows the basic outline  as the proof of 
Theorem 9.1 in \cite{Sturmfels(1996)}. For any even closed walk $\Gamma=(i_1,i_2,...,i_{2k-1},i_{2k},i_1)$ in $K_n^\circ$ we associate the binomial 
\[
b_\Gamma:= \prod_{l=1}^k\sigma_{i_{2l-1},i_{2l}}-\prod_{l=1}^k\sigma_{i_{2l},i_{2l+1}}
\]
which belongs to $SP_{K_n^\circ}$.
To prove that $S'$ is a Gr{\"o}bner basis, it is enough to prove that the initial 
monomial of any binomial $b_\Gamma$ is divisible by some monomial 
$\sigma_{ij}\sigma_{kl}$ which is the initial term of some binomial in 
$S'$, where $(i,j)$ and $(k,l)$ are a pair of non intersecting 
edges.
Let there exist a binomial $b_\Gamma=\sigma^u-\sigma^v\in SP_{K_n^\circ}$ with 
$in_\prec(b_\Gamma)=\sigma^u$ which contradicts the assertion. 
Then assuming that $b_\Gamma$ has minimal weight, 
we can say that each pair of edges appearing in $\sigma^v$ intersects.

The edges of the walk are labeled as even or odd, where even edges look like 
$(i_{2r},i_{2r+1})$ and the odd edges are of the form $(i_{2r-1},i_{2r})$.
We pick an edge $(s,t)$ of the walk $\Gamma$ which has the least circular 
distance between $s$ and $t$. The edge $(s,t)$ separates the vertices 
of $K_n^\circ$ except $s$ and $t$ into two disjoint 
sets $P$ and $Q$ where 
$|P|\geq |Q|$. We start $\Gamma$ at $(s,t)=(i_1,i_2)$. 
From our assertion on $b_\Gamma$ we have that each pair of odd 
(resp. even) edges intersect. Also, it can be proved that if 
$P$ contains an odd vertex $i_{2r-1}$, then it contains all 
the subsequent odd vertices $i_{2r+1},i_{2r+3},...,i_{2k-1}$. 
As the circular distance between $s$ and $t$ is the least, 
we need to have $i_3$ to be in $P$. So, all the odd vertices 
except $i_1$ lie in $P$ and all the even vertices lie in 
$Q\cup\{i_1,i_2\}$. This gives us that the two even edges 
$(i_2,i_3)$ and $(i_{2k},i_1)$ do not intersect, which is a contradiction.  
\end{proof}

Our goal next is to use Lemma \ref{lem:added loops}, to
prove that $SP_G =  CI_G$ for block graphs with one central vertex.
Recall that the set $D$  consisted of all pairs $\sigma_{ij}$ such that in the graph $G$
$i \leftrightarrow j$ does not touch the central vertex.
As the $\sigma_{ij}$ appearing in $D$ do not appear in any generators of $SP_G$, 
let us construct an associated subgraph of $K_n^\circ$ without
those edges.  Specifically, let $G^\circ$ be the graph obtained by removing 
the edges $(i,j)$ from $K_n^\circ$ such that  $\sigma_{ij} \in D$.
Note that we choose an embedding of $G^\circ$ so that each
maximal clique minus $c$ forms a contiguous block on the circle.
The placement of $c$ can be anywhere that is between the maximal blocks.
 
Figure \ref{fig:3graphs} illustrates the construction of the graph $G^\circ$ 
in an example.

\begin{ex}\label{ex:three graphs}
Let $G$ be a block graph with 5 vertices in Figure \ref{fig:3graphs}. There are 3 possible 1-clique partitions of $G$, each of them having $C=\{3\}$. The edges in $K_5^\circ$ which cannot be separated by any 1-clique partition of $G$ are $D=\{(1,2),(1,1),(2,2),(4,4),(5,5)\}$. So we remove them from $K_5^\circ$ to get $G^\circ$.

\begin{figure}
\begin{tikzpicture}
\filldraw[black]
(0,1) circle [radius=.04] node [below] {3}
(-1,2) circle [radius=.04] node [above] {1}
(-1,0) circle [radius=.04] node [below] {2}
(1,2) circle [radius=.04] node [above] {4}
(1,0) circle [radius=.04] node [below] {5}
(0,-.75) circle [radius=0] node [below] {G};
\draw
(-1,0)--++(2,2)
(-1,0)--++(0,2)
(1,0)--++(-2,2);

\draw
(4,-0.2) circle [radius=.2] node {\hspace{-1cm}3}
(6,-0.2) circle [radius=.2] node {\hspace{1cm}4}
(3.9,2.16) circle [radius=.2] node {\hspace{-1cm}2}
(6.1,2.16) circle [radius=.2] node{\hspace{1cm}5}
(5,2.7) circle [radius=.2] node [above] {\hspace{.5cm}1}
(5,-.75) circle [radius=0] node [below] {$K_5^\circ$}; 

\filldraw[black]
(4,0) circle [radius=.04] 
(6,0) circle [radius=.04] 
(4,2) circle [radius=.04] 
(6,2) circle [radius=.04]
(5,2.5) circle [radius=.04]; 

\draw
(4,0)--++(2,0)
(4,0)--++(2,2)
(4,0)--++(0,2)
(4,0)--++(1,2.5)
(4,2)--++(1,.5);
\draw[dotted]
(4,2)--++(2,0)
(6,0)--++(0,2)
(6,0)--++(-2,2)
(6,2)--++(-1,.5)
(6,0)--++(-1,2.5);

\filldraw[black]
(9,0) circle [radius=.04] node {\hspace{-1cm}}
(11,0) circle [radius=.04] node {\hspace{1cm}4}
(9,2) circle [radius=.04] node {\hspace{-1cm}2}
(11,2) circle [radius=.04] node{\hspace{1cm}5}
(10,2.5) circle [radius=.04] node [above] {\hspace{.5cm}1}
(10,-.75) circle [radius=0] node [below] {$G^\circ$}; 
\draw
(9,-0.2) circle [radius=.2] node {\hspace{-1cm}3}
(9,0)--++(2,0)
(9,0)--++(2,2)
(9,0)--++(0,2)
(9,0)--++(1,2.5);
\draw[dotted]
(9,2)--++(2,0)
(11,0)--++(0,2)
(11,0)--++(-2,2)
(11,2)--++(-1,.5)
(11,0)--++(-1,2.5);
\end{tikzpicture}
\caption{Construction of the graph $G^\circ$. The dark lines in $K_5^\circ$ correspond to the edges in $G$ whereas a dotted line between $i$ and $j$ tells us that there is no edge between $i$ and $j$ in $G$. The dotted line basically corresponds to the shortest path between the two vertices in $G$.  Note that the addition of extra
edges gives us $K_5^\circ$ and the deletion of some edges gives us $G^\circ$. \label{fig:3graphs}}
\end{figure}
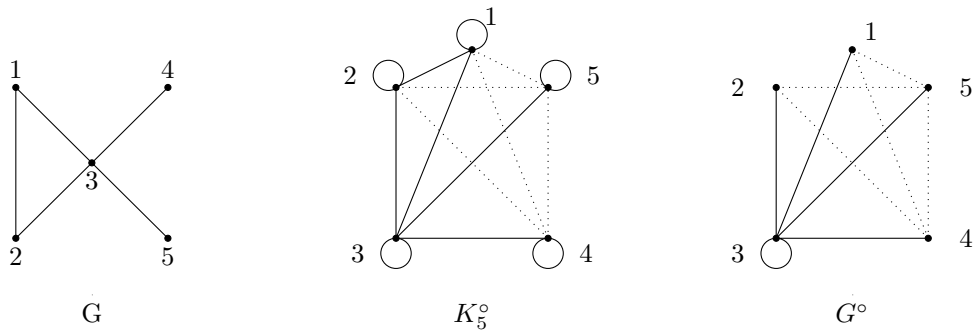
\end{ex}

\begin{lem}\label{lem:separation by 1-clique partition}
For any non intersecting 
pair of edges $(i,j),(k,l)$ in $G^\circ$, 
there exists a 1-clique partition $(A,B,C)$ of $G$ such 
that $i,l\in A\cup C$ and $j,k\in B\cup C$. 
\end{lem} 

\begin{proof}
We first prove this for the non intersecting 
edges $(i,j),(k,l)$ with 
$i,j,k,l\neq c$. Without loss of generality we can assume that 
$i<j<k<l$. We know that for each edge $(i,j)$ in $G^\circ$ there 
exists a 1-clique partition $(A,B,C)$ of $G$ such that 
$i\in A\cup C$ and $j\in B\cup C$. This implies that 
$i$ and $j$ (similarly $k$ and $l$) lie in different maximal 
cliques of $G$. As the vertices of $G^\circ$ are labeled 
counter-clockwise, there are only three ways how the 
vertices $i,j,k,l$ can be placed:
\begin{eqnarray*}
&i)& i,l\in C_1, j,k\in C_2, \hspace{.6cm}ii) \hspace{.4cm}i,l\in C_1, j\in C_2, k\in C_3,\\
&iii)& i\in C_1, j\in C_2,k\in C_3,l\in C_4,
\end{eqnarray*}
where $C_i$ are the different maximal cliques of $G$. In all the three cases $i$ and $k$ (similarly $j$ and $l$) are in different maximal cliques. Hence there exists a 1-clique partition $(A,B,C)$ such that $i,l\in A\cup C$ and $k,j\in B\cup C$.

A similar argument can be given for the non intersecting 
edges $(i,c),(k,l)$ and $(c,c),(i,j)$. 
\end{proof}

\begin{lem}\label{lem:elimination}
Let $S'$ be the Gr{\"o}bner basis for 
$SP_{K_n^\circ} \subseteq \mathbb{R}[\Sigma]$ 
as defined in Lemma \ref{lem:added loops}. 
Then the set $S'\cap \mathbb{R}[\Sigma \setminus D]$ 
forms a Gr{\"o}bner basis for $SP_G$.
\end{lem}

\begin{proof}
Let $g=\sigma^u-\sigma^v$ be an arbitrary binomial in 
$SP_G = \ker \hat \psi$.
This implies that the initial term of 
$g$ is contained in $\mathbb{R}[\Sigma \setminus D]$. 
Since $S'$ is a Gr{\"o}bner basis for $SP_{K_n^\circ}$ with respect to $\prec$, 
there must exist some $f\in S'$ such that ${\rm in}_{\prec}(f)$ divides 
$in_{\prec}(g)$. This gives us that the initial term of 
$f$ is contained in $\mathbb{R}[\Sigma \setminus D]$.

So it is enough to show that for every $f\in S'$ whose leading term 
is in $\mathbb{R}[\Sigma \setminus D]$ is actually contained in 
$\mathbb{R}[\Sigma \setminus D]$. Let 
\[
f=\sigma_{ij}\sigma_{kl}-\sigma_{ik}\sigma_{jl}
\]
be a binomial in $S'$ whose leading term is contained in 
$\mathbb{R}[\Sigma \setminus D]$. 
Let $\sigma_{ij}\sigma_{kl}$ be the leading term. 
Then the edges $(i,j), (k,l)$ are  non intersecting 
as the initial term 
of each binomial in $S'$ corresponds to the non intersecting 
edges. 
So by Lemma \ref{lem:separation by 1-clique partition}, there must exist a 1-clique partition $(A,B,C)$ of $G$ which separates the edges 
$(i,j)$ and $(k,l)$, that is, $i,l\in A\cup C$ and $j,k\in B\cup C$.
 This implies that $(A,B,C)$ also separates the edges $(i,k)$ and $(j,l)$. 
 Hence we can say that $\sigma_{ik},\sigma_{jl}\notin D$ and 
 $\sigma_{ij}\sigma_{kl}-\sigma_{ik}\sigma_{jl}\in\mathbb{R}[\Sigma \setminus D]$.
\end{proof}

Now that we have all the required results, we prove the main result of this section.

\begin{thm}\label{thm:main1}
Let $G$ be a block graph with $n$ vertices having only one central vertex. 
Then the set of all $2\times 2$ minors of 
$\Sigma_{A\cup C, B\cup C}$ for all possible 1-clique partitions 
$(A,B,C)$ of $G$ form a Gr{\"o}bner basis for $SP_G$. In particular, $SP_G=CI_G$.
\end{thm}

\begin{proof}
We rearrange the graph by placing the vertices in $K_n^\circ$ 
such that there is no intersection among the edges of $G$ in 
$A\cup C$ and $B\cup C$ for any 1-clique partition $(A,B,C)$ 
(with $C=\{c\}$). We complete the graph by drawing the remaining 
edges with dotted lines.

The complete graph $K_n^\circ$ gives us a partial term order on 
$\mathbb{R}[\Sigma]$ by defining the weight of the variable 
$\sigma_{ij}$ as the number of edges of $K_n^\circ$ which do
not intersect 
the edge $(i,j)$. Let $\prec$ denote the term order 
that refines the partial order on monomials specified by the weights. 
Let $S$ be the set of all $2\times 2$ minors of $\Sigma_{A\cup C,B\cup C}$ 
for all possible 1-clique partitions of $G$. Any binomial in $S$ has 
one of the three forms: 
\begin{enumerate}[i)]
\item  $\sigma_{ij}\sigma_{kl}-\sigma_{ik}\sigma_{jl}$ 
 with  $ i,l\in A\cup C \text{ and } j,k\in B\cup C $  
\item  $\sigma_{ij}\sigma_{kl}-\sigma_{il}\sigma_{jk}$ 
 with $ i,k\in A\cup C \text{ and } j,l\in B\cup C $ 
\item  $\sigma_{il}\sigma_{jk}-\sigma_{ik}\sigma_{jl} $
with  $ i,j\in A\cup C \text{ and } k,l\in B\cup C. $
\end{enumerate}
Here $(i,j),(k,l)$ and $(i,l),(j,k)$ are the non intersecting 
pairs of edges and $(i,k)(j,l)$ is the intersecting pair in $G^\circ$. 
So any binomial in $S$ of the form $(i)$ or $(iii)$ is contained in $S'$. 
If the binomial $\sigma_{ij}\sigma_{kl}-\sigma_{il}\sigma_{jk}$ 
(of form $(ii)$) is in $S$, then by Lemma 
\ref{lem:separation by 1-clique partition} we know that the binomials 
$\sigma_{ij}\sigma_{kl}-\sigma_{ik}\sigma_{jl}$ and
 $\sigma_{il}\sigma_{jk}-\sigma_{ik}\sigma_{jl}$ are also in $S$. As
\[
\sigma_{ij}\sigma_{kl}-\sigma_{il}\sigma_{jk}=
\sigma_{ij}\sigma_{kl}-\sigma_{ik}\sigma_{jl}-(\sigma_{il}\sigma_{jk}-\sigma_{ik}\sigma_{jl}),
\]
we can conclude that $S$ and $S\cap S'$ generate the same ideal.
Furthermore, the set $S \cap S'$ has the same initial terms as
$S' \cap \rr[\Sigma \setminus D]$  so this guarantees that
$S$ is a Gr\"obner basis for $SP_G$ as well.
\end{proof}


\section{The shortest path ideal for an arbitrary block graph}  \label{sec:allgraphs}

To generalize the statement in Theorem \ref{thm:main1} 
for any arbitrary block graph, we further exploit
the toric structure of the ideal $SP_G$.
As $SP_G$ is the kernel of a monomial map, it is a toric ideal,
a prime ideal generated by binomials.  Finding a generating set of 
$SP_G$ is equivalent to finding a set of binomials that make some associated
graphs connected.  We use this perspective to prove that $SP_G = CI_G$.

From the shortest path map 
$\psi$, we can obtain the matrix 
$M_\psi$ as shown in Example \ref{ex:Mphi}. 
So $SP_G = \ker(\psi)$ is the toric ideal of the matrix $M_\psi$ as 
\[
\psi(\sigma^u)=t^{M_\psi u},
\]
where $\sigma=(\sigma_{11},\sigma_{12},...,\sigma_{nn})$ and 
$t=(a_1,a_2,...,a_n,k_{12},...,k_{n-1n})$.

Let $G=([n],E)$ be a block graph. For any vector $b\in \mathbb{N}^{(n+|E|)}$, 
the \emph{fiber} of $M_\psi$ over $b$ is defined as 
\[
M_\psi^{-1}(b)=\{u\in \mathbb{N}^{(n^2+n)/2}:M_\psi u=b\}.
\]
As the columns of $M_\psi$ are non-zero and non-negative, $M_\psi^{-1}(b)$
is always finite for any $b\in \mathbb{N}^{(n+|E|)}$. 
Let $\mathcal{F}$ be any finite subset of $\ker_\zz(M_\psi)$. 
The \emph{fiber graph}  $M_\psi^{-1}(b)_\mathcal{F}$ is defined as follows:
\begin{enumerate}[i)]
\item The nodes of this graph are the elements of $M_\psi^{-1}(b)$. 
\item Two nodes $u$ and $u'$ are connected by an edge if $u-u'\in \mathcal{F}$
 or $u'-u\in \mathcal{F}$. 
\end{enumerate}

The fundamental theorem of Markov bases connects the generating
sets of toric ideals to connectivity properties of the fiber graphs. 
We state this explicitly in the case of the fiber graphs for the shortest
path maps.

\begin{thm}(Thm 5.3, \cite{Sturmfels(1996)}\label{thm:connected-generators})
Let $\mathcal{F}\subset \ker_\zz(M_\psi)$. 
The graphs $M_\psi^{-1}(b)_\mathcal{F}$ are connected for all 
$b \in \mathbb{N}M_\psi= \\ \{\lambda_1M_{\psi 1}+...+\lambda_{n+|E|}M_{\psi n+|E|} : \lambda_i \in \mathbb{N}, M_{\psi i} \text{ are columns of } M_\psi\}$ if and only if the set 
$\{\sigma^{v^+}-\sigma^{v^-}:v\in \mathcal{F}\}$ generates the toric ideal $SP_G$.
\end{thm}

As we proved in Theorem \ref{thm:main1} that the set of all 
$2\times 2$ minors of $\Sigma_{A\cup C,B\cup C}$ for all possible 
1-clique partitions of $G$ form a Gr{\"o}bner basis for $\ker(\psi)$ 
for all block graphs with one central vertex, by using Theorem 
\ref{thm:connected-generators} we can say that the graph 
$M_\psi^{-1}(b)_\mathcal{F}$ is connected for 
all $b\in \mathbb{N}M_\psi$. Here $\mathcal{F}$ is the set of all 
$2\times 2$ minors of $\Sigma_{A\cup C,B\cup C}$ in the vector form, 
for all possible 1-clique partitions of $G$. 

So, to generalize the result in Theorem \ref{thm:main1} for all block graphs, 
we need to show that $M_\psi^{-1}(b)_\mathcal{F}$ is connected for any 
$b\in \mathbb{N}M_\psi$. For a fixed $b$, let $u,v\in M_\psi^{-1}(b)_\mathcal{F}$. 
This implies that both $M_\psi u $ and $M_\psi v$ are equal to $b$,
which gives us $\psi(\sigma^u-\sigma^v)=0$. Therefore, it is enough to show 
that for any $f=\sigma^u-\sigma^v\in SP_G$, 
$\sigma^u$ and $\sigma^v$ are connected by the moves in $\cf$.

Let $G$ be a block graph with $n$ vertices. 
Let $u\in \mathbb{N}^{(n^2+n)/2}$ which is a node in the graph of $M_{\psi}^{-1}(b)_\mathcal{F}$.  We represent this $u$, or equivalently $\sigma^u$, as a graph in the following way:
For each factor $\sigma_{ij}$ of $\sigma^u$ we draw the shortest path 
$i\leftrightarrow j$ along $G$ with end points at $i$ and $j$. 
For each $\sigma_{ii}$ we draw a loop at the vertex $i$. 
Let $\deg_i(\sigma^u)$ denote the \emph{degree} of a vertex $i$ in 
$\sigma^u$ which is defined to be the number of end points of paths in $\sigma^u$.
We count the loops corresponding to $\sigma_{ii}$ as having two endpoints at $i$. 

If $f=\sigma^u-\sigma^v$ is a homogeneous binomial in $SP_G$, then 
$\psi(\sigma^u)=\psi(\sigma^v)$ if and only if the following conditions are
satisfied:

\begin{enumerate}[i)]
\item The graphs of $\sigma^u$ and $\sigma^v$ both have the 
same number of paths (as $f$ is homogeneous), 
\item The graphs of $\sigma^u$ and $\sigma^v$ have the same number 
of edges between any two adjacent vertices $i$ and $j$ 
(as the exponent of $k_{ij}$ in $\psi(\sigma^u)$ gives the number 
of edges between $i$ and $j$ in the graph of $\sigma^u$), 
\item   The degree of any vertex in both the graphs is the same 
(as the exponent of $a_i$ in $\psi(\sigma^u)$ gives us the degree 
of the vertex $i$ in the graph of $\sigma^u$). 
\end{enumerate}

Next we show how to use the results from Section \ref{sec:1central}
to make moves that bring $\sigma^u$ and $\sigma^v$ closer together.
This approach works by localizing the computations at each central
vertex in the graph.

Let $c$ be a central vertex in $G$. 
We define a map $\rho_c$ between the set of vertices as follows:
\[ \rho_c(i) = \begin{cases} 
          c & i=c \\
          i &  i \text{ is adjacent to } c \\
          i' & i' \text{ is adjacent to } c \text{ and lies in } i\leftrightarrow c.
       \end{cases}
    \]
Let $G_c$ be the graph obtained by applying $\rho_c$ to the vertices of $G$. 
Note that $G$ can have multiple vertices mapped to a single vertex in $G_c$. 
The map $\rho_c$ can also be seen as a map between 
$\rr[\Sigma]$ to itself by the rule 
$\rho_c(\sigma_{ij})=\sigma_{\rho_c(i)\rho_c(j)}$.

For a vector $u \in \nn^{n(n+1)/2}$ and $c$ a central vertex 
let $u_c$ be the vector that extracts all the coordinates that correspond
to shortest paths that touch $c$.  That is, 
\[
u_c(ij)= \begin{cases}
         u(ij) & c\in i\leftrightarrow j\\
         0 & \text{otherwise}.
         \end{cases}
    \]

\begin{prop} \label{prop:psi-c1}
Suppose that $\sigma^u - \sigma^v \in SP_G$ and let $c$ be a central
vertex of $G$.  Then
$\psi_{G_c}(\rho_c(\sigma^{u_c}))-\psi_{G_c}(\rho_c(\sigma^{v_c}))=0$.
\end{prop}

Note that we use the notation $\psi_{G_c}$ to denote that 
we use the $\psi$ map associated to the graph $G_c$. However, the map
$\psi$ associated to $G$ can be used since that will give the same result. 

\begin{proof}
We have
\[ \rho_c(\sigma_{ij})=\begin{cases}
\sigma_{ij} & i,j \text{ are adjacent to } c\\
\sigma_{ic} & i \text{ is adjacent to } c, j=c \\
\sigma_{cj} & j \text{ is adjacent to } c, i=c\\
\sigma_{i'c} & i' \text{ is adjacent to } c \text{ and } i'\in i\leftrightarrow c,  j=c\\
\sigma_{cj'} & j' \text{ is adjacent to } c \text{ and } j'\in j\leftrightarrow c, i=c\\

\sigma_{ij'} & i,j' \text{ are adjacent to } c \text{ and } j'\in c\leftrightarrow j\\
\sigma_{i'j} & i',j \text{ are adjacent to } c \text{ and } i'\in i\leftrightarrow c \\
\sigma_{i'j'} & i',j' \text{ are adjacent to } c \text{ and } i'\in i\leftrightarrow c, j'\in j\leftrightarrow c_1 \\
\sigma_{i'i'} & i' \text{ is adjacent to } c \text{ and } i' \in i\leftrightarrow c \text{ and } j\leftrightarrow c. 
\end{cases}
\]
We know that $\sigma^u$ and $\sigma^v$ have the same number of paths. 
Also, the degree of each vertex and the number of edges between 
any two adjacent vertices is the same.
So, it is enough to show that $\rho_c(\sigma^{u_c})$ and 
$\rho_c(\sigma^{v_c})$ have the same number of paths and the degree of each vertex, 
number of edges between any two adjacent vertices is also the same.
\begin{eqnarray*}
\text{ Number of paths in } \sigma^{u_c} &=& 
\text{ number of paths in } \sigma^u \text{ ending at } c +\\
&&\text{ number of paths containing } c \text{ but not ending at } c \\
&=& \text{ degree of } a_{c} \text{ in } \psi(\sigma^u) + 1/2 (\text{ number of variables of } \\
&& \text{ the form } k_{ic} \text{ in } \psi(\sigma^u) - \text{ degree of } a_{c} \text{ in } \psi(\sigma^u ) ) \\
&=& \text{ number of paths in } \sigma^{v_c}
\end{eqnarray*}
The number of paths in $\sigma^{u_c}$ and $\rho_c(\sigma^{u_c})$ 
are the same as $\rho_c$ maps monomials of degree 1 to monomials of degree 1.

For any vertex $s$ which adjacent to $c$, the degree of $s$ in $\rho_c(\sigma^{u_c})$ is 
\begin{eqnarray*}
\deg_{s}(\rho_c(\sigma^{u_c}))&=& \text{ number of edges } s\leftrightarrow c 
\text{ in } \sigma^{u} \\
&=& \text{ number of edges } s\leftrightarrow c \text{ in } \sigma^{v} \\
&=& deg_{s}(\rho_c(\sigma^{v_c})).
\end{eqnarray*}

Now, for any two vertices $i'$ and $j'$ adjacent to $c$, 
the number of edges $i'\leftrightarrow j'$ in $\rho_c(\sigma^{u_c})$ 
is $0$ as every path in $\rho_c(\sigma^{u_c})$ contains $c$. 
The number of edges $i'\leftrightarrow c$ in $\rho_c(\sigma^{u_c})$ 
is equal to the number of edges $i'\leftrightarrow c$ in $\sigma^u$, 
which is equal to the number of edges $i'\leftrightarrow c$ in $\sigma^v$.

Hence, we can conclude that 
$
\psi_{G_c}(\rho_c(\sigma^{u_c}))-\psi_{G_c}(\rho_c(\sigma^{v_c}))=0.$
\end{proof}

By Theorem \ref{thm:connected-generators} we know that we can reach 
from $\rho_c(\sigma^{u_c})$ to $\rho_c(\sigma^{v_c})$ by making a finite set 
of moves from the set of $2\times 2$ minors of $\Sigma_{A\cup C,B\cup C}$, 
for all possible 1-clique partitions of $G_c$. But from the map $\rho_c$ 
we have that for each move $\sigma_{i'j'}\sigma_{k'l'}-\sigma_{i'l'}\sigma_{k'j'}$ 
in $G_c$ there exists a corresponding move $\sigma_{ij}\sigma_{kl}-\sigma_{il}\sigma_{kj}$ 
in $G$, where $i'\leftrightarrow j'\subseteq i\leftrightarrow j$ and 
$k'\leftrightarrow l'\subseteq k\leftrightarrow l$.  
In fact, there are many such corresponding moves corresponding to all the ways
to pull back $\rho_c$.  

\begin{defn}
Let $G$ be a block graph and let $c$ be a central vertex.  
We call two monomials $\sigma^u$ and $\sigma^v$ in the same fiber
to be \emph{similar at a vertex $c$ } if the subgraph over 
$c$ and its adjacent vertices is the same for both the monomials.
\end{defn}

For a given block graph $G$ and a central vertex $c$, let $S_c$
denote the set of all $2 \times 2$ minors of all matrices $\Sigma_{A \cup C, B \cup C}$
where $(A,B,C)$ is a separation condition that is valid for $G$ with
$C = \{c\}$.

\begin{prop} \label{prop:move correspondence}
If a sequence of moves in $G_c$ take $\rho_c(\sigma^{u_c})$ to $\rho_c(\sigma^{v_c})$, 
then there exist a corresponding sequence of moves in $S_c$ which takes $\sigma^u$ 
to a monomial which is similar to $\sigma^v$ at $c$.
\end{prop}

\begin{proof}
We know that $\rho_c(\sigma^{u_c})$ and $\sigma^u$ are similar at $c$ by construction.
So, it is enough to show that if $m$ is a move in $G_c$ and $m'$ is the 
corresponding move in $G$, then $m$ applied to $\rho_c(\sigma^{u_c})$ and $m'$ applied to $\sigma^u$ are similar at $c$. 
Let $m=\sigma_{i'j'}\sigma_{k'l'}-\sigma_{i'l'}\sigma_{k'j'}$ 
be a move in $G_c$ acting on the paths $\sigma_{i'j'},\sigma_{k'l'}$ 
in $\rho_c(\sigma^{u_c})$. 
Let $m'=\sigma_{ij}\sigma_{kl}-\sigma_{il}\sigma_{kj}$ be its corresponding 
move in $S_c$ acting on the paths $\sigma_{ij},\sigma_{kl}$ in $\sigma^u$. 
As $i'\leftrightarrow j'\subseteq i\leftrightarrow j$, 
$k'\leftrightarrow l'\subseteq k\leftrightarrow l$ and $c\in i'\leftrightarrow j'$ 
and $k'\leftrightarrow l'$, $m$ and $m'$ make the same changes at $c$ in both 
the graphs. So, we can conclude that $m$ applied to $\rho_c(\sigma^{u_c})$
and $m'$ applied to $\sigma^u$ are similar at $c$.
\end{proof}

Once we have the set of moves which takes $\sigma^u$ to a monomial which is 
similar to $\sigma^v$ at $c$, we can apply the same procedure at the
other central vertices as well. To show that this
ends up producing two monomials that are similar at every central vertex
it is necessary to check that the moves obtained for a different 
central vertex $c'$ do not affect the structure previously obtained at $c$.  

\begin{prop} \label{prop:similarity preserved}
Let $m=\sigma_{ij}\sigma_{kl}-\sigma_{il}\sigma_{kj}$ be a move 
obtained from a partition with $C=\{c\}$. Let $V$ be the set of vertices in $G$. 
Then $\sigma^u$ and $m$ applied to $\sigma^u$ are similar at 
$V\setminus c$.
\end{prop}

\begin{proof}
If $s$ is any vertex which is not in $i\leftrightarrow j$ or 
$k\leftrightarrow l$, then $\sigma^u$ and $m$ applied to $\sigma^u$ remain similar at 
$s$ as the move does not make any change at $s$. 
If $s\neq c$ is a vertex in $i\leftrightarrow j$, we then consider 2 cases:

Case 1:  $s\in i\leftrightarrow j$ and $s\notin k\leftrightarrow l$ 

Let $s\in i\leftrightarrow c$. As $m$ converts $i\leftrightarrow c \leftrightarrow j$ 
to $i \leftrightarrow c \leftrightarrow l$, 
$i \leftrightarrow c$ is contained in $i \leftrightarrow l$. 
This implies that $s$ and all the vertices in $i\leftrightarrow j$ adjacent to $s$ 
are also present in $i\leftrightarrow l$. A similar argument
applies for $s\in c\leftrightarrow j$.

Case 2: $s\in i\leftrightarrow j$ and $s\in k\leftrightarrow l$

Let $s\in i\leftrightarrow c$ and $s\in k\leftrightarrow c$. 
As $m$ converts $i\leftrightarrow c \leftrightarrow j$ to 
$i \leftrightarrow c \leftrightarrow l$ and $k\leftrightarrow c \leftrightarrow l$ 
to $k \leftrightarrow c \leftrightarrow j$, 
$i \leftrightarrow c$ is contained in $i \leftrightarrow l$ and $k \leftrightarrow c$ 
is contained in $k \leftrightarrow j$. 
So $s$ and all the vertices in $i\leftrightarrow j$ 
($k\leftrightarrow l$) adjacent to $s$ are present in 
$i\leftrightarrow l$ ($k\leftrightarrow j$). 
A similar argument applies for $s\in c\leftrightarrow j, c\leftrightarrow l$.

In both the cases, $m$ preserves the structure of 
$\sigma^u$ around the vertex $s$. Hence, $\sigma^u$ and $m$ applied to $\sigma^u$ 
are similar at all the vertices in $V\setminus c$. 
\end{proof}

Note an important key feature that follows from the proof of
Proposition \ref{prop:similarity preserved}:
If $m$ can be obtained from two partitions $(A_1, B_1, C_1)$ and $(A_2, B_2, C_2)$
with different central vertices, then $\sigma^u$ and 
$m$ applied to $\sigma^u$ are similar at the central vertices as well.

We now give a proof for the generalized version of Theorem \ref{thm:main1}.

\begin{thm}\label{thm: ker(psi) = J_G}
Let $G$ be a block graph. Then the shortest path ideal $SP_G$
is generated by the set of all $2\times 2$ minors of 
$\Sigma_{A\cup C,B\cup C}$, for all possible 1-clique partitions of $G$, i.e, $SP_G=CI_G$.
\end{thm}

\begin{proof}
Suppose that $c_1, \ldots, c_k$ are the central vertices of $G$.
Let $S_1, \ldots S_k$ be the corresponding quadratic moves associated to
each central vertex.  
Let $f=\sigma^u-\sigma^v \in SP_G$.
By applying Proposition \ref{prop:move correspondence} and Proposition
\ref{prop:similarity preserved} together with Theorem \ref{thm:main1},
we can assume that $\sigma^u$ and $\sigma^v$ are similar at every vertex
after applying moves from $S_1, \ldots, S_k$.

 We can 
assume that $\sigma^u$ and $\sigma^v$ have no variables in common,
otherwise we could delete this variable from both monomials and do an induction 
on dimension.
So consider an arbitrary path $i\leftrightarrow j$ in 
$\sigma^{u}$ which is not present in $\sigma^v$. 
We select the path in $\sigma^v$ which has the highest number 
of common edges with $i\leftrightarrow j$. 
Let that path be $i'\leftrightarrow j'$ and let 
$s\leftrightarrow t $ be the common path in both the paths. 
Let $s_1$ and $t_1$ be the vertices adjacent to 
$s$ and $t$ respectively in $i\leftrightarrow j$. Similarly, 
let $s'$ and $t'$ be the vertices adjacent to $s$ and $t$ 
respectively in $i'\leftrightarrow j'$. Let $p$ be the vertex in 
$s\leftrightarrow t$ adjacent to $t$ (see Figure \ref{fig:path moves}
for an illustration of the idea). 

If we apply the map $\rho_t$ on both the monomials, 
we get that there exists a path $p\leftrightarrow t_1$ 
in $\rho_t(\sigma^{u})$ which is not in $\rho_t(\sigma^{v})$. 
But as $\sigma^{u}$ and $\sigma^v$ are similar at $t$, 
there must exist a path $x\leftrightarrow y$ in 
$\sigma^v$ containing $p\leftrightarrow t_1$. 
So, the move $m=\sigma_{i'j'}\sigma_{xy}-\sigma_{i'y}\sigma_{xj'}$ 
is a valid move as none of the vertices in 
$i'\leftrightarrow p$ can be adjacent to any vertex in $t_1\leftrightarrow y$ 
(as it would form a closed circuit implying that $i'\leftrightarrow t$ 
is not the shortest path). Similarly, none of the vertices in $x\leftrightarrow p$ 
can be adjacent to any vertex in $t'\leftrightarrow j'$. 
Further, this move can be obtained from two different partitions 
with central vertices $p$ and $t$ respectively. 
So, by Proposition \ref{prop:similarity preserved} and the comment after its proof, 
we know that the move $\sigma_{i'j'}\sigma_{xy}-\sigma_{i'y}\sigma_{xj'}$
preserves the similarity of all the vertices.

Applying $m$ on $\sigma^v$ increases the length of the 
common path between $i\leftrightarrow j$ and $i'\leftrightarrow j'$ by at least 1, 
while keeping the monomials $\sigma^{u}$ and $m$ applied to $\sigma^v$ similar at all the vertices. Repeating this process again, we can continue to shorten the length of
the disagreement until the resulting monomials have a common monomial,
in which case induction implies that we can use moves to connect these smaller
degree monomials.

This implies that the set of binomials $S_1 \cup \cdots \cup S_k$ generates
$SP_G$ and hence $CI_G =  SP_G$.

\begin{figure}
\begin{tikzpicture}
\filldraw[black]
(0,0) circle [radius=.04] node [below] {$i'$}
(1,0) circle [radius=.02] 
(1.15,0) circle [radius=.02] 
(1.3,0) circle [radius=.02] 
(1.45,0) circle [radius=.02]
(2,0) circle [radius=.02] node [below] {$s'$}
(3,0) circle [radius=.02] node [below] {$s$}
(3.15,0) circle [radius=.02] 
(3.3,0) circle [radius=.02]  
(3.45,0) circle [radius=.02] 
(4,0) circle [radius=.02] node [below] {$p$}
(5,0) circle [radius=.02] node [below] {$t$}
(6,0) circle [radius=.02] node [below] {$t'$}
(6.15,0) circle [radius=.02] 
(6.3,0) circle [radius=.02] 
(6.45,0) circle [radius=.02] 
(6.6,0) circle [radius=.02]
(7,0) circle [radius=.04] node [below] {$j'$}
(4,-1) circle [radius=0] node [below] {$\sigma^v$}
(2,1.75) circle [radius=.04] node [above] {$x$}
(4,.25) circle [radius=.02]
(5,.25) circle [radius=.02]
(5.66,.87) circle [radius=.03] node [below] {$t_1$}
(6.6,1.75) circle [radius=.04] node [above] {$y$}
(2.28,1.34) circle [radius=.02]
(2.587,.87) circle [radius=.03] 
(2.4,1.15) circle [radius=.02]
(5.95,1.15) circle [radius=.02]
(6.15,1.34) circle [radius=.02]

(8,0) circle [radius=.04] node [below] {$i$}
(9,0) circle [radius=.02] 
(9.15,0) circle [radius=.02] 
(9.3,0) circle [radius=.02] 
(9.45,0) circle [radius=.02]
(10,0) circle [radius=.02] node [below] {$s_1$}
(11,0) circle [radius=.02] node [below] {$s$}
(11.15,0) circle [radius=.02] 
(11.3,0) circle [radius=.02]  
(11.45,0) circle [radius=.02] 
(12,0) circle [radius=.02] node [below] {$p$}
(13,0) circle [radius=.02] node [below] {$t$}
(14,0) circle [radius=.02] node [below] {$t_1$}
(14.15,0) circle [radius=.02] 
(14.3,0) circle [radius=.02] 
(14.45,0) circle [radius=.02] 
(14.6,0) circle [radius=.02]
(15,0) circle [radius=.04] node [below] {$j$}
(3,.25) circle [radius=.02]
(11.5,-1) circle [radius=0] node [below] {$\sigma^{u}$};

\draw
(0,0)--++(1,0)
(1.45,0)--++(.55,0)
(2,0)--++(1.15,0)
(3.45,0)--++(2.7,0)
(6.6,0)--++(.4,0)
(2,1.75)--++(1,-1.5)
(3,.25)--++(2,0)
(5,.25)--++(1.6,1.5)

(8,0)--++(1,0)
(9.45,0)--++(.55,0)
(10,0)--++(1.15,0)
(11.45,0)--++(2.7,0)
(14.6,0)--++(.4,0);

\end{tikzpicture}
\caption{\label{fig:path moves}}
\end{figure}

\end{proof}


\section{Initial term map and SAGBI bases}\label{sec:SAGBI}

In this section we put all our previous
results on shortest path maps together to prove Theorem \ref{thm: theConjecture}.
We also show that the set of polynomials $\{f_{ij} : 1 \leq i \leq j \leq n \}$
obtained from the inverse of $K$ are a SAGBI basis for the $\rr$-algebra they
generate in the case of block graphs.

\begin{proof}[Proof of Theorem \ref{thm: theConjecture}]
We have already seen that $SP_G = CI_G \subseteq P_G$.  
We just need to show that $SP_G = P_G$ to complete the proof.
Note that both $SP_G$ and $P_G$ are prime ideals so it suffices to show
that they have the same dimension.  

In both $SP_G$ and $P_G$ an upper bound on the
dimension is equal to the number of vertices plus the number of edges
in the graph.  This follows because that is the number of free
parameters in both parametrizations.  
In the case of $P_G$ this upper bound is tight, because the map that sends
$\Sigma \mapsto \Sigma^{-1}$ is the inverse map that recovers the entries of $K$. 
Since $SP_G \subseteq P_G$ we have the $\dim SP_G \geq  \dim P_G$.  
Hence they must have the same dimension. 
\end{proof}

Finally, we can show the SAGBI basis property for the polynomials 
$\{f_{ij} : 1 \leq i \leq j \leq n \}$.  Recall the definition of a
SAGBI basis (which stands for Subalgebra Analogue of Gr\"obner Basis for Ideals).  
See Chapter 11 of \cite{Sturmfels(1996)} for more details.

\begin{defn}
Let $R$ be a finitely generated subalgebra of the polynomial ring 
$\mathbb{R}[K]$. Let $\prec$ be a term order 
on $\mathbb{R}[K]$.  The \emph{initial algebra} ${\rm in}_\prec (R)$ is defined as the 
$\mathbb{R}$-vector space spanned by $\{{\rm in}_\prec (f):f\in R\}$.
A finite set of polynomials $F \subseteq R$ is called a SAGBI basis for $R$ if
\begin{enumerate}[i)]
\item  $R =  \rr[F]$, and
\item  ${\rm in}_\prec (R)  =   \rr[  \{{\rm in}_\prec (f):f\in F\}]$. 
\end{enumerate}
\end{defn}

Let $G$ be a block graph and let $F = \{f_{ij} : 1 \leq i \leq j \leq n \}$ be the polynomials
appearing as the numerators in $K^{-1}$.  To prove this, we will use some key result on
SAGBI bases.  Note that if $\prec$ is a term order on $\rr[K]$ induced by
a weight vector $\omega$, then 
this induces a partial term order on $\rr[\Sigma]$ by declaring that the weight
of the variable $\sigma_{ij}$ is the weight of the largest monomial appearing
in $f_{ij}$.  Denote by $\omega^*$ this induced weight order on $\rr[\Sigma]$.  

Both the algebras $\rr[F]$ and $\rr[  \{{\rm in}_\prec (f):f\in F\}]$ have
presentation ideals in $\rr[\Sigma]$.
In the first case, this presentation ideal is exactly $P_G$, the vanishing ideal of
the Gaussian graphical model.   That is,  $\rr[F]  =  \rr[\Sigma]/P_G$. 
In the second case, this presentation is exactly $SP_G$, the shortest path ideal,
since that is the ideal of relations among the shortest path monomials.
That is, $\rr[  \{{\rm in}_\prec (f):f\in F\}] =  \rr[\Sigma]/SP_G$.

A fundamental theorem on SAGBI bases applied in the specific case of these ideals
says the following.

\begin{thm} (Thm 11.4, \cite{Sturmfels(1996)}\label{thm:CanonicalBasis})
The set $F \subseteq \rr[K]$ is a SAGBI basis if and only if ${\rm in}_{\omega^*} (P_G)  =  SP_G$.
\end{thm}

\begin{cor}
Let $G$ be a block graph.
Then the set $F \subseteq \rr[K]$ is a SAGBI basis of $\rr[F]$.  
\end{cor}

\begin{proof}
We have already shown that $SP_G = P_G$.  By construction, every
one of the binomials in $SP_G$ is homogeneous with respect to the weighting $\omega^*$.  
Indeed, this weighting is exactly the weighting that counts the multiplicity of each edge of $\sigma^u$
and the $\deg_i(\sigma^u)$ as used in Section \ref{sec:allgraphs}.
But then ${\rm in}_{\omega^*} (P_G) = {\rm in}_{\omega^*} (SP_G) = SP_G$ as desired.
By Theorem \ref{thm:CanonicalBasis}, this shows that $F$ is a SAGBI basis.
\end{proof}



\bibliographystyle{apalike}

\begin{thebibliography}{}

\bibitem[DeLoera et~al., 1995]{DeLoera(1995)}
DeLoera, J.~A., Sturmfels, B., and Thomas, R.~R. (1995).
\newblock Gr\"{o}bner bases and triangulations of the second hypersimplex.
\newblock {\em Combinatorica}, 15(3):409--424.

\bibitem[Dobra and Sullivant, 2004]{Dobra(2004)}
Dobra, A. and Sullivant, S. (2004).
\newblock A divide-and-conquer algorithm for generating {M}arkov bases of
  multi-way tables.
\newblock {\em Computational Statistics}, 19(3):347--366.

\bibitem[Drton et~al., 2008]{Drton(2008)}
Drton, M., Massam, H., and Olkin, I. (2008).
\newblock Moments of minors of {W}ishart matrices.
\newblock {\em The Annals of Statistics}, 36(5):2261--2283.

\bibitem[Drton et~al., 2007]{Drton(2007)}
Drton, M., Sturmfels, B., and Sullivant, S. (2007).
\newblock Algebraic factor analysis: tetrads, pentads and beyond.
\newblock {\em Probability Theory and Related Fields}, 138(3-4):463--493.

\bibitem[Grayson and Stillman, 2017]{Grayson(2017)}
Grayson, D.~R. and Stillman, M.~E. (2017).
\newblock Macaulay2, a software system for research in algebraic geometry.
\newblock \url{ https://faculty.math.illinois.edu/Macaulay2/}.

\bibitem[Hassett, 2007]{Hassett(2007)}
Hassett, B. (2007).
\newblock {\em Introduction to algebraic geometry}.
\newblock Cambridge University Press, Cambridge.

\bibitem[Herzog et~al., 2018]{Herzog(2018)}
Herzog, J., Hibi, T., and Ohsugi, H. (2018).
\newblock {\em Binomial ideals}, volume 279 of {\em Graduate Texts in
  Mathematics}.
\newblock Springer, Cham.

\bibitem[Ho\c{s}ten and Sullivant, 2002]{Hosten(2002)}
Ho\c{s}ten, S. and Sullivant, S. (2002).
\newblock Gr\"{o}bner bases and polyhedral geometry of reducible and cyclic
  models.
\newblock {\em Journal of Combinatorial Theory. Series A}, 100(2):277--301.

\bibitem[Jones and West, 2005]{Jones(2005)}
Jones, B. and West, M. (2005).
\newblock Covariance decomposition in undirected {G}aussian graphical models.
\newblock {\em Biometrika}, 92(4):779--786.

\bibitem[Kelley, 1935]{Kelley(1935)}
Kelley, T. (1935).
\newblock {\em Essential Traits of Mental Life}, volume~26 of {\em Harvard
  Studies in Education}.
\newblock Harvard University Press, Cambridge, MA.

\bibitem[Koller and Friedman, 2009]{Koller(2009)}
Koller, D. and Friedman, N. (2009).
\newblock {\em Probabilistic graphical models}.
\newblock Adaptive Computation and Machine Learning. MIT Press, Cambridge, MA.

\bibitem[Lauritzen, 1996]{Lauritzen(1996)}
Lauritzen, S.~L. (1996).
\newblock {\em Graphical models}, volume~17 of {\em Oxford Statistical Science
  Series}.
\newblock Oxford University Press, New York.

\bibitem[Spirtes et~al., 2000]{Spirtes(2000)}
Spirtes, P., Glymour, C., and Scheines, R. (2000).
\newblock {\em Causation, prediction, and search}.
\newblock Adaptive Computation and Machine Learning. MIT Press, Cambridge, MA,
  second edition.

\bibitem[Sturmfels, 1996]{Sturmfels(1996)}
Sturmfels, B. (1996).
\newblock {\em Gr\"{o}bner bases and convex polytopes}, volume~8 of {\em
  University Lecture Series}.
\newblock American Mathematical Society, Providence, RI.

\bibitem[Sturmfels and Uhler, 2010]{Sturmfels(2010)}
Sturmfels, B. and Uhler, C. (2010).
\newblock Multivariate {G}aussian, semidefinite matrix completion, and convex
  algebraic geometry.
\newblock {\em Annals of the Institute of Statistical Mathematics},
  62(4):603--638.

\bibitem[Sullivant, 2007]{Sullivant(2007)}
Sullivant, S. (2007).
\newblock Toric fiber products.
\newblock {\em Journal of Algebra}, 316(2):560--577.

\bibitem[Sullivant, 2008]{Sullivant(2008)}
Sullivant, S. (2008).
\newblock Algebraic geometry of {G}aussian {B}ayesian networks.
\newblock {\em Advances in Applied Mathematics}, 40(4):482--513.

\bibitem[Sullivant, 2018]{Sullivant(2018)}
Sullivant, S. (2018).
\newblock {\em Algebraic statistics}, volume 194 of {\em Graduate Studies in
  Mathematics}.
\newblock American Mathematical Society, Providence, RI.

\bibitem[Sullivant et~al., 2010]{Sullivant(2010)}
Sullivant, S., Talaska, K., and Draisma, J. (2010).
\newblock Trek separation for {G}aussian graphical models.
\newblock {\em The Annals of Statistics}, 38(3):1665--1685.

\end{thebibliography}

\nocite{*}


\end{document}